\documentclass{article}
\usepackage{amsfonts}
\usepackage{amsmath}
\usepackage{color}
\usepackage{graphicx}
\usepackage{enumerate}
\usepackage{multirow}
\usepackage{booktabs}
\usepackage{verbatim}
\usepackage{cases}

\usepackage{epstopdf}
\usepackage{epsfig}
\usepackage{subfigure}

\usepackage{cite}

\usepackage{amsthm}

\newtheorem{theorem}{Theorem}
\newtheorem{proposition}{Proposition}

\newtheorem{lemma}{Lemma}

\newtheorem{assumption}{Assumption}
\theoremstyle{remark}
\newtheorem{remark}{Remark}

\newcommand{\my}{\mathbf{y}}
\newcommand{\me}{\mathbf{e}}
\newcommand{\mz}{\mathbf{z}}
\newcommand{\ox}{\overline{x}}
\newcommand{\oy}{\overline{y}}
\newcommand{\oV}{\overline{V}}

\begin{document}

\title{A Flocking-based Approach for Distributed Stochastic Optimization\thanks{To appear in Operations Research. Copyright: \copyright\,2017 INFORMS. The authors gratefully acknowledge partial support from AFOSR (FA9550-15-1-0504) and NSF (1561381).}}
\author{
$\begin{array}{ccc}
\text{Shi Pu}\thanks{Shi Pu is with the Department of Industrial and Systems Engineering, University of Florida, Gainesville, FL 32611-6595, e-mail: shipu@ufl.edu} &  & \text{Alfredo Garcia}\thanks{Alfredo Garcia is with the Department of Industrial and Systems Engineering, Texas A\&M University, College Station, TX 77843-3131, e-mail: {agarcia@ise.ufl.edu}} \\ 
%\text{University of Florida} &  & \text{University of Florida}
\end{array}
$
}
\maketitle
\abstract{
In recent years, the paradigm of cloud computing has emerged as an architecture for computing that makes use of distributed (networked) computing resources. In this paper, we consider a distributed computing algorithmic scheme for stochastic optimization which relies on modest communication requirements amongst processors and most importantly, does not require synchronization. Specifically,  we analyze a scheme with $N>1$ independent threads 
implementing each a stochastic gradient algorithm. The threads are coupled via a perturbation of the gradient (with attractive and repulsive forces) in a similar manner to mathematical models of flocking, swarming and
other group formations found in nature with mild communication requirements.
When the objective function is convex, we show that a flocking-like approach for distributed stochastic optimization provides a noise reduction effect similar to that of a centralized stochastic gradient algorithm based upon the average of $N$ gradient samples at each step. The distributed nature of flocking makes it an appealing computational alternative. We show that when the overhead related to the time needed to gather $N$ samples and synchronization is not negligible, the flocking implementation outperforms a centralized stochastic gradient algorithm based upon the average of $N$ gradient samples at each step.
When the objective function is not convex, the flocking-based approach seems better suited to escape locally optimal solutions due to the repulsive force which enforces a certain level of diversity in the set of candidate solutions. Here again, we show that the noise reduction effect is similar to that associated to the centralized stochastic gradient algorithm based upon the average of $N$ gradient samples at each step.
}

\section{Introduction}

\label{opt sec:intro}

Swarms, flocks and other group formations can be found in nature in many
organisms ranging from simple bacteria to mammals (see \cite{parrish2002self,okubo1986dynamical,reynolds1987flocks} for references). Such collective and
coordinated behavior is believed to be effective for avoiding predators and/or for
increasing the chances of finding food (foraging) (see \cite{grunbaum1998schooling,pu2016noise}). In this paper
we introduce a novel distributed scheme for stochastic optimization wherein multiple independent computing threads implement each a stochastic gradient algorithm  which is further perturbed by repulsive and attractive terms (a function of the relative
distance between solutions). Thus, the updating of individual solutions
is coupled in a similar manner to mathematical models of swarming, flocking and
other group formations found in nature (see \cite{gazi2011swarm}).
We show that this coupling endows the flocking scheme with an important robustness
property as noise
realizations that induce trajectories differing too much from the group
average are likely to be discarded.

The performance of the single-thread stochastic gradient algorithm is highly sensitive to noise. Thus, there is a literature on estimation techniques leading to better gradient estimation often involving increasing sample size (see \cite{fu2015stochastic,spall2005introduction} for a survey of gradient
estimation techniques). When sampling is undertaken in parallel, synchronization is needed to execute the tasks that can not be executed in parallel.
The speed-up obtained by parallel sampling and {\em centralized} gradient estimation is limited by overhead related to {\em (i)} time spent gathering samples (which could be significant for example in the simulation of complex systems) and {\em (ii)} synchronization. In contrast, the noise reduction obtained in a flocking-based approach with $N>1$ threads does not require synchronization since each thread only needs the information on the current solution identified by neighboring threads (where the notion of neighborhood is related to a given network topology).
When sampling times are not negligible and exhibit large variation, synchronization may cause significant overhead so that  {\em real-time} performance of stochastic gradient algorithm based upon the average of $N$ samples obtained in parallel is highly affected by large sampling time variability. In contrast, the {\em real-time} performance of a flocking-based implementation with $N>1$ threads may be superior as each thread can asynchronously update its solution based upon a  {\em small} sample size and {\em still} reap the benefits of noise reduction stemming from the flocking discipline.

To illustrate the noise reduction property, consider the minimization of the function $f(x)=\ln(\|x\|^2+1)$ where $x \in \mathbb{R}^2$. The unique optimal solution is $x^*=(0,0)$. Suppose that the gradient $\nabla f(x)$ is observed with noise so that the basic iteration in a stochastic gradient descent algorithm can be written as:
\begin{equation*}
x({k+1})=x({k})+\Gamma(k)(-\nabla f(x({k}))+\varepsilon({k})),
\end{equation*}
where  $\Gamma({k})>0$ is the step size, and the collection $\{\varepsilon({k}):k>0\}$ is i.i.d. (independent and identically distributed) with mean zero and variance $\sigma^2$.
The stochastic gradient method with constant step size $\Gamma(k)=0.02$ and normally distributed noise with $\sigma^2=450$ is unable to approximate the optimal solution given the large magnitude of noise (relative to the gradient).
%Even when a solution that is close to optimal is reached the algorithm is not robust to noise perturbation. This is true even for the {\em best} performing thread. 
One solution to this conundrum is to implement an improved version based upon the average of $N=10$ samples of the gradient at each step. 
Then the variance of noise is reduced to ${\sigma}^2/N=45$. Supposing that each step takes $0.02s$, the performance of this approach is shown in Figure \ref{opt fig: illustration}(a)\footnote{In this illustration example we assume zero variance among sampling times.}.

In this paper we advocate a different tack. We introduce an additional perturbation to the gradient so that the basic iteration for thread $i$ is:
\begin{equation*}
x({i,k+1})=x({i,k})+\tilde{\Gamma}({i,k})\left[-\nabla f(x({i,k}))+\varepsilon ({i,k})-\sum_{j=1,j\neq
	i}^{N}\alpha_{ij}\nabla_{x(i,k)} J(\|x({i,k})-x({j_i,k})\|)\right],
\end{equation*}
where $ J(\|x({i,k})-x(j_i,k)\|)$ represents the {\em flocking} potential between threads $i$ and $j$. $\alpha_{ij}=1$ if thread $i$ has access to the current solution $x({j_i,k})$ identified by thread $j$ and $\alpha_{ij}=0$ otherwise.
The term $\nabla_{x(i,k)} J(\|x({i,k})-x({j_i,k})\|)$ is a combination of \emph{repulsive} and \emph{attractive} ``forces" depending upon the relative distance $\|x({i,k})-x({j_i,k})\|$ (see \cite{gazi2003stability} for reference). The performance of the flocking-based scheme is measured by the average solution of all threads.
%{\color{blue} We use the average of all $x_{i,k}$'s as the final output. \color{red} I think we should state clear what is the output of the flocking-based algorithm.}

Figure \ref{opt fig: illustration}(b) depicts the performance of the flocking-discipline of $N=10$ fully connected threads (again with constant step size $\tilde{\Gamma}({i,k})=0.02$ and noise variance $\sigma^2=450$).
It can be seen to be comparable with the scheme based upon the average of $N=10$ samples at each step.
%(see Figure \ref{opt fig: illustration}(b)). 
The times needed for the solutions identified by each scheme to reach the ball $B_{0.1}(x^*)=\{x\in \mathbb{R}^2~|~ \|x-x^*\|\ < 0.1\}$ are fairly similar \footnote{$x^*$ denotes the optimal solution. For a total of $100$ sample paths
	the mean time and standard deviation of the centralized scheme and the flocking scheme are $(992.6, 1454.6)$ and $(969.1, 1223.4)$, respectively.}.

This noise reduction effect can be succinctly explained as follows. Under a flocking discipline, noise
realizations that induce trajectories differing too much from the group
average are likely to be discarded because of the attractive potential effect on each individual thread which leads to cohesion. The noise reduction enabled by a flocking-discipline is fundamentally different from that associated with the averaging of independent gradient samples. 

\begin{figure}[htbp]
	\centering
	\subfigure[Sample paths for sample average scheme with $N$ samples.]{\includegraphics[width=3.0in]{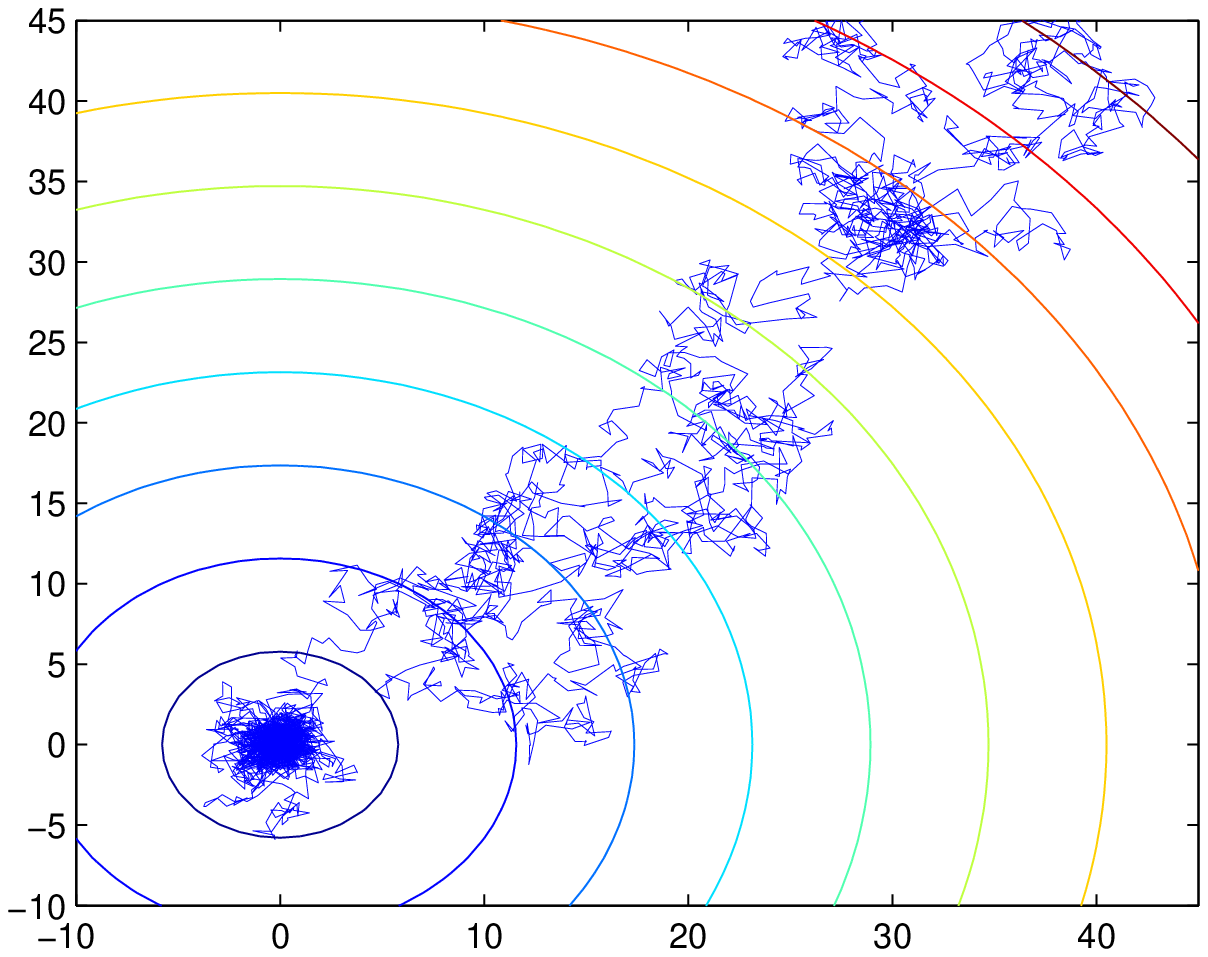}} 
	\subfigure[Sample paths with flocking discipline ($N$ threads).]{\includegraphics[width=3.0in]{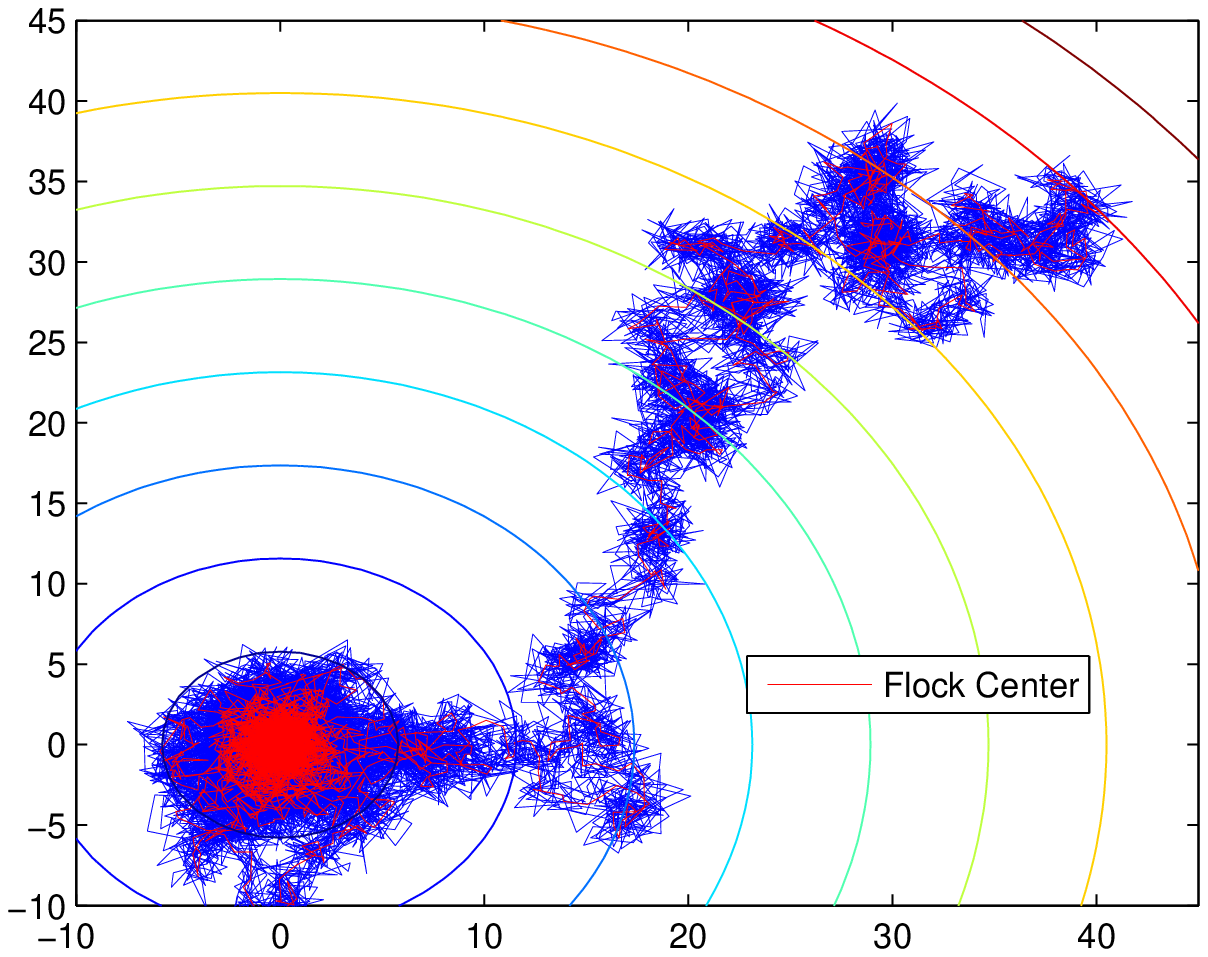}} 
	\subfigure[Distance to optimum for sample average scheme with $N$ samples.]{\includegraphics[width=3.0in]{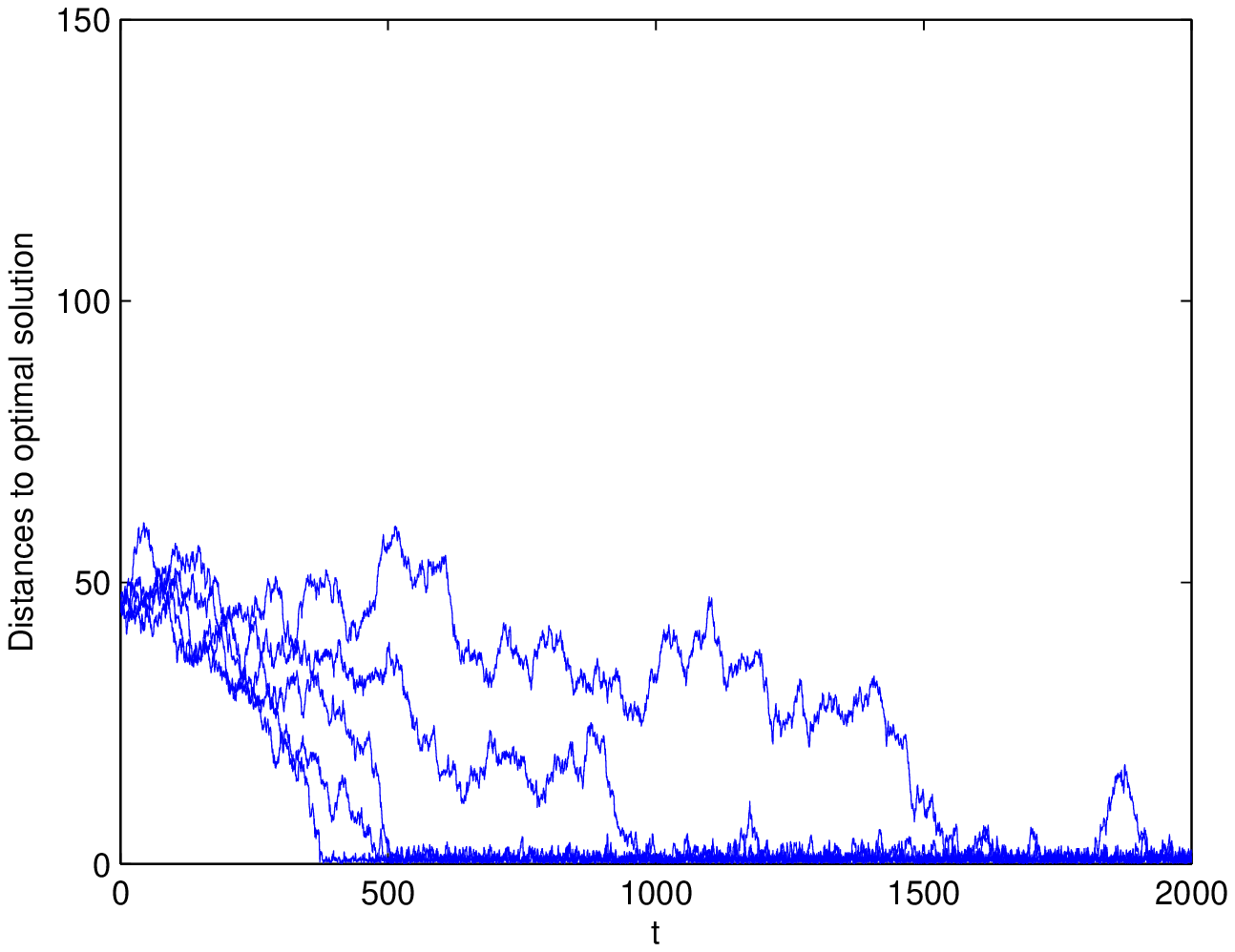}}
	\subfigure[Distance between the average flocking solution and the optimum.]{\includegraphics[width=3.0in]{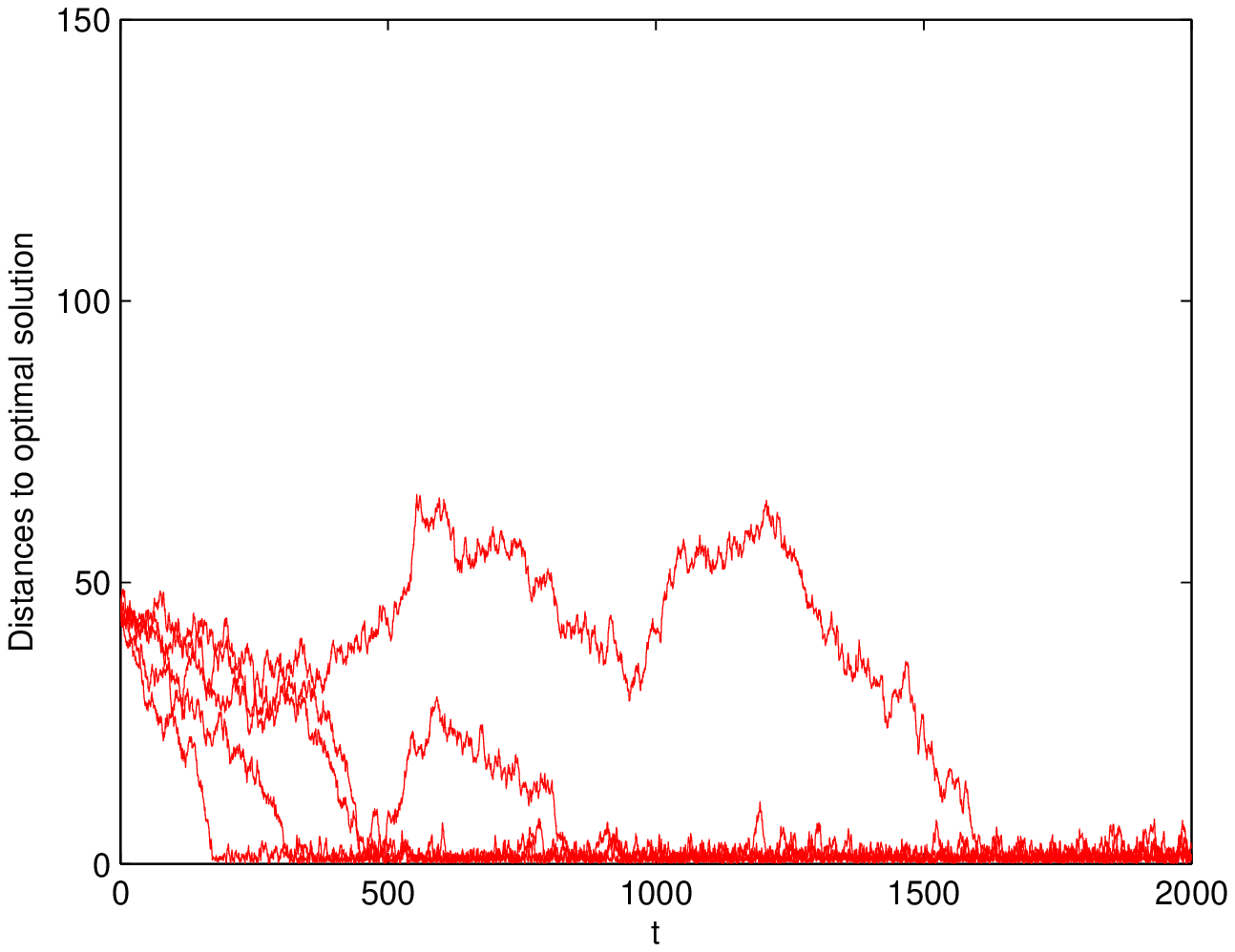}} 
	\caption{Performance comparison between the scheme based upon the average of $N=10$ gradient samples per step and the flocking-based approach with $10$ fully connected threads.}
	\label{opt fig: illustration}
\end{figure}

Our work is related to the extensive literature in stochastic approximation method dating to \cite{robbins1951stochastic} and \cite{kiefer1952stochastic}. These work includes the analysis of convergence (conditions for convergence, rates of convergence, proper choice of step size) in the context of diverse noise models (see \cite{kushner2003stochastic}). 
%Recent progress include \cite{ghadimi-a,ghadimi-b,roux} and \cite{hennig}. 
Recently there has been considerable interest in parallel or distributed implementation of stochastic gradient algorithms (see \cite{cavalcante2013distributed,towfic2014adaptive,lobel2011distributed,srivastava2011distributed,wang2015cooperative} for examples). However, they mainly aim at minimizing a sum of convex functions which is different from our objective.

%{\color{blue}Our work is also related to particle swarm optimization (PSO) which is a population-based algorithm designed by \cite{eberhart1995new}. In recent years, several noise mitigation mechanisms have been incorporated into PSO to tackle optimization problems subject to noise \cite{cui2005tracking,pan2006particle}. Incorporation of such mechanisms has shown to provide improvements on the quality of the solutions. Nevertheless, no theoretical result was given that guarantees convergence to optimal/suboptimal solutions.}

Our work is also linked with population-based algorithms for simulation-based
optimization. In these approaches, at every iteration, the quality of each
solution in the population is evaluated and a new population of solutions is
randomly generated according to a given rule designed to achieve an
acceptable trade-off between \textquotedblleft exploration\textquotedblright\
and \textquotedblleft exploitation\textquotedblright\ effort. Recent efforts
have focused on model-based algorithms (see \cite{hu2007model}) which differ
from population-based approaches in that candidate solutions are generated
at each iteration by sampling from a \textquotedblleft
model\textquotedblright\ which is a probability distribution over the
solution space. The basic idea is to modify the model based on the sampled
solutions in order to bias the future search towards regions containing high
quality solutions (see \cite{hu2015model} for a recent survey). These approaches are inherently
{\em centralized} in that the updating of populations (or models) is undertaken
after the quality of {\em all} candidate solutions is evaluated.

The structure of this paper is as follows. Section \ref{opt sec: set} introduces the optimization problem of interest. In Section \ref{opt sec:cohesion}, we perform cohesion analysis of the flocking-based approach with respect to the solutions identified by different threads. Section \ref{opt sec: increased convex} formalizes the noise reduction properties of the flocking-based algorithmic scheme for convex optimization. In Section \ref{opt sec: nonconvex}, we apply the flocking-based algorithm  to the optimization of general non-convex functions.  Section \ref{opt sec: conclusion} concludes the paper.

\section{Setup}
\label{opt sec: set}
\subsection{Preliminaries}
\label{opt subsec:graph}
In the analysis of this paper we shall make use of certain graph theoretic concepts which we briefly review below.
A graph $\mathcal{G}$ is a pair $(\mathcal{V},\mathcal{E})$, where $\mathcal{V}$ is a set of vertices and $\mathcal{E}$ is a subset of $\mathcal{V}\times \mathcal{V}$ called edges. A graph $(\mathcal{V},\mathcal{E})$ is called undirected if $(i,j)\in \mathcal{E}\Leftrightarrow (i,j)\in \mathcal{E}$. The adjacency matrix $A=[\alpha_{ij}]\in \mathbb{R}$ of a graph is a matrix with nonzero elements satisfying the property $\alpha_{ij}>0\Leftrightarrow (i,j)\in \mathcal{E}$. Self-joining edges are excluded, i.e., $\alpha_{ii}=0,\forall i$. 
%	The set of neighbors of vertex $i$ is defined by
%	\begin{align}
%	\mathcal{N}_i=\{j\in \mathcal{V}: a_{ij}\neq 0\}.
%	\end{align}
%	
%	
The Laplacian matrix $L$ associated with a graph $(\mathcal{V},\mathcal{E})$ is defined as $L=[l_{ij}]$, where $l_{ii}=\sum_j \alpha_{ij}$ and $l_{ij}=-\alpha_{ij}$ where $i\neq j$. For an undirected graph, the Laplacian matrix is symmetric positive semi-definite (see \cite{godsil2013algebraic}).
%The Kronecker product of matrices $A\in\mathcal{R}^{n\times n}$ and $B\in\mathcal{R}^{m\times m}$ is an $mn\times mn$ block matrix defined as follows:
%\begin{align*}
%A\otimes B=\begin{pmatrix}
%a_{11}B & a_{12}B & \ldots & a_{1n}B\\
%a_{21}B & a_{22}B & \ldots & a_{2n}B\\
%\vdots & \vdots & \ddots & \cdots\\
%a_{n1}B & a_{n2}B & \ldots & a_{nn}B\\
%\end{pmatrix}.
%\end{align*}
\subsection{Problem Statement}

We consider the problem%
\begin{equation}
\min_{x\in \mathbb{R}^{m}}f(x)  \label{opt Problem_def}
\end{equation}%
where $f:\mathbb{R}^{m}\rightarrow \mathbb{R}$ is a differentiable function
that is not available in closed form.
% and $X\subseteq \mathbb{R}^{m}$. 
To solve this problem, a black-box noisy simulation model is used. In this
context, noise can have many sources such as modeling and discretization
error, incomplete convergence, and finite sample size for Monte-Carlo
methods (see for instance \cite{kleijnen2008design}). Assume that we have $N$ computing threads that can generate gradient samples in parallel. 
%The time needed for each unit to generate one gradient sample is approximately $\Delta t$.
Every gradient sample  is subject to i.i.d. noise $\varepsilon\in \mathbb{R}^m$ of mean zero and variance $\sigma ^{2}$ in each dimension.

In the rest of this section, we present two algorithms for solving the problem. First, we introduce a centralized stochastic gradient-descent algorithm. Then we propose the flocking-based approach. 
In both cases, we assume the step size is a constant value ($\Gamma$ and $\tilde{\Gamma}$, respectively).

\subsubsection{A Centralized Algorithm}
A {\em centralized} stochastic gradient-descent algorithm is of the form:%
\begin{equation}
x(k+1)=x(k)+\Gamma u(k), \ \ k\in \mathbb{N} \label{opt eq:x_k centralized}
\end{equation}%
where 
%$\Gamma$ is the step-size, and 
$u(k)=-\nabla f(x(k))+\epsilon (k)$, with $\epsilon (k)$ being the random simulation noise. Assume that sampling takes place through $N$ parallel computing threads, where each thread contributes one sample in a single step. Then, $\epsilon (k)$ is given by an average of $N$ i.i.d. random vectors: $\epsilon (k)= (1/N)\sum_{i=1}^{N}\varepsilon (i)$. Each $\varepsilon (i)$ has $m$ i.i.d. components of mean zero and variance $\sigma ^{2}$.

In what follows, we approximate the discrete-time process (\ref{opt eq:x_k centralized}) by a continuous-time system for ease of analysis. 
Let $\Delta t(k)$ be the time needed to gather $N$ samples for calculating $u(k)$. Then solution $x(k)$ is obtained at $t(k)=\sum_{l<k}\Delta t(l)$ in continuous-time. Denoting by $x_t$ the identified solution at time $t$, we have $x_{t(k)}=x(k)$
%, and that $\Delta t :=(t(k)-t(k-1))$ is the approximate sampling and processing time between two contiguous steps. 
%Notice that $\Delta t$ is typically greater than $\Delta t$ due to the overhead of parallel computing. 
and $x_t=x(k-1)$ for all $t\in(t(k-1),t(k))$, i.e., the continuous-time solution changes discretely.
% only upon updates. 
By scheme (\ref{opt eq:x_k centralized}),
%\begin{align*}
%x_{t(k)}=x_0-\sum_{m<k}\nabla f(x_{t(l)})\Gamma+\sum_{m<k}\epsilon(l)\Gamma,
%\end{align*}
%or
\begin{align*}
x_{t}=x_0-\sum_{t(l)<t}\nabla f(x_{t(l)})\Gamma+\sum_{t(l)<t}\epsilon(l)\Gamma.
\end{align*}
%Notice that $x_t$ is discontinuous at $t=t(l),\forall l$. 
%\begin{align*}
%x_{t}=x_0-\sum_{t(j)<t}\nabla f(x_{t(j)})\Gamma\Delta t\frac{1}{\Delta t}+\sum_{t(j)<t}\epsilon_{t(j)}\Gamma\Delta t\frac{1}{\Delta t}.
%\end{align*}
Define a new variable $y_t:=x_{t/\Gamma}$. It follows that
%\begin{align*}
%y_{\Gamma t}=x_0-\sum_{t(j)<t/\Gamma}\nabla f(x_{\Gamma t(j)})\Gamma\Delta t\frac{1}{\Delta t}+\sum_{t(j)<t}\epsilon_{t(j)}\Gamma\Delta t\frac{1}{\Delta t}.
%\end{align*}
\begin{align*}
y_{t}=  x_0-\sum_{t(l)<t/\Gamma}\nabla f(x_{t(l)})\Gamma+\sum_{t(l)<t/\Gamma}\epsilon(l)\Gamma= y_0-\sum_{\Gamma t(l)<t}\nabla f(y_{\Gamma t(l)})\Gamma+\sum_{t(l)<t/\Gamma}\epsilon(l)\Gamma.
\end{align*}
%Denote by $n_t$ the cardinality of $\{j:\Gamma t(j)<t\}$.
Assume that all $\Delta t(k)$'s are i.i.d. with mean $\mathbb{E}[\Delta t(k)]=\Delta t$, and let $n_t$ be the cardinality of $\{l: \Gamma t(l)<t\}$. By the (strong) law of large numbers, for small $\Gamma>0$,
\begin{align}
\frac{\Gamma n_t}{t}=\frac{n_t}{t/\Gamma}\simeq \frac{1}{\Delta t},\; \text{and} \; \frac{n_t}{t}\gg 0.
\label{opt n_t}
\end{align}
Hence,
\begin{align*}
\sum_{\Gamma t(l)<t}\nabla f(y_{\Gamma t(l)})\Gamma=\frac{\Gamma n_t}{t}\left[\sum_{\Gamma t(l)<t}\nabla f(y_{\Gamma t(l)})\frac{t}{n_t}\right]\simeq \frac{\Gamma n_t}{t}\int_{0}^{t}\nabla f(y_t)dt\simeq \frac{1}{\Delta t}\int_{0}^{t}\nabla f(y_t)dt.
%=\sum_{\Gamma t(l)<t}\nabla f(y_{\Gamma t(l)})\Gamma\Delta t\frac{1}{\Delta t}
\end{align*}

Note that
$\sum_{t(l)<t/\Gamma}\epsilon(l)\Gamma$ has mean zero and variance $\Gamma^2 n_t\sigma^2/N $ in each dimension. In light of (\ref{opt n_t}),
\begin{align*}
\frac{\Gamma^2 n_t\sigma^2}{N} =\frac{\Gamma n_t}{t}\Delta t\frac{\Gamma\sigma^2 t}{N\Delta t}\simeq \frac{\Gamma\sigma^2}{N\Delta t} t.
\end{align*}
We have
\begin{align*}
\sum_{t(l)<t/\Gamma}\epsilon(l)\Gamma\simeq \sigma\sqrt{\frac{\Gamma}{N\Delta t}}B_t,
\end{align*}
where $B_{t}$ is the standard $m$-dimensional Brownian motion.

Define $\gamma=1/\Delta t$ and $\tau_N=\sigma\sqrt{{\Gamma \Delta t}/{N }}$.
Then $y_t$ approximately satisfies the following stochastic Ito integral:
\begin{align*}
y_t=y_0-\gamma\int_{0}^{t}\nabla f(y_t)dt+\int_{0}^{t} \tau_N\gamma dB_t.
\end{align*}
which is usually written in its differential form:
\begin{equation}
dy_{t}=-\nabla f(y_t) \gamma dt+\tau_N\gamma dB_t.  \label{opt eq:dy_t centralized}
\end{equation}%

%\begin{remark}
%	$\{y_t:t\ge 0\}$ is a rescaled process of $\{x_t:t\ge 0\}$.
%\end{remark}
%{\color{blue} We assume that $\gamma_t=\gamma$ is constant throughout the paper, which implies that the step size $\Gamma_t=\Gamma$ is also constant.}

\subsubsection{A Flocking-based Algorithm}
A {\em flocking-based} implementation also has $N$ computing threads. In contrast to the centralized approach, each thread $i$ independently implement a  stochastic gradient algorithm based on only {\em one} sample at each step:
\begin{equation}
x(i,k+1)=x(i,k)+\tilde{\Gamma} u(i,k), \ \ k\in \mathbb{N}
\label{opt eq: x(i,k)}
\end{equation}
where
\begin{align*}
u(i,k)=-\nabla f(x(i,k))+\varepsilon ({i,k})-\sum_{j=1,j\neq
	i}^{N}\alpha_{ij}\nabla_{x(i,k)} J(\|x(i,k)-x(j_i,k)\|).
\end{align*}
Here $x(j_i,k)$ denotes the current solution of thread $j\neq i$ at the time of thread $i$'s update, and
noise term $\varepsilon ({i,k})$ comes from one sampling. Thus each $\varepsilon ({i,k})$ is i.i.d. with $m$ i.i.d. components of mean zero and variance $\sigma^2$. 

The additional term $-\sum_{j=1,j\neq
	i}^{N}\alpha_{ij}\nabla_{x_i} J(\|x(i,k)-x(j_i,k)\|)$ represents the function of mutual attraction and repulsion between individual threads (see \cite{gazi2004class} for reference).
$A=[\alpha_{ij}]\in \mathbb{R}^{N\times N}$ is the coupling matrix with $\alpha_{ij}\in\{0,1\}$. $\alpha_{ij}=1$ indicates that thread $i$ is informed of the solution identified by threads $j$. We assume that the corresponding graph $\mathcal{G}$ is undirected ($A=A^T$) and connected. 

%Denote by $x({i,k})$ the solution of thread $i$ after step $k$. 
Denote with $\Delta t({i,k})$ the time needed by thread $i$ to gather one sample for $u(i,k)$, then $x({i,k})$ is obtained at time $t(i,k)=\sum_{l<k}\Delta t({i,l})$. Let $x_{i,t}$ be the solution of thread $i$ at time $t$. It satisfies $x_{t(i,k)}=x(i,k)$ and $x_{i,t}=x(i,k-1)$ for all $t\in(t(i,k-1),t(i,k))$.  
The scheme can be written as follows: for each thread $i\in \{1,\ldots ,N\}$,
\begin{equation*}
%\label{opt eq: x_i flock}
x_{i,t(i,k+1)}=x_{i,t(i,k)}+\tilde{\Gamma}\left[-\nabla f(x_{i,t(i,k)})+\varepsilon ({i,k})-\sum_{j=1,j\neq
	i}^{N}\alpha_{ij}\nabla_{x_{i,t(i,k)}} J(\|x_{i,t(i,k)}-x_{j,t(i,k)}\|)\right].
\end{equation*}%

Define function $g(\cdot)$ as $g(x)=-\nabla_{x} J(\|x\|)$. Let $y_{i,t}=x_{i,t/\tilde{\Gamma}}$ and assume all $\Delta t({i,k})$'s are i.i.d. with $\mathbb{E}[\Delta t({i,k})]=\tilde{\Delta} t$. Similar to (\ref{opt eq:dy_t centralized}), for small $\tilde{\Gamma}>0$ the dynamics of $y_{i,t}$ can be approximated by%
\begin{equation}
dy_{i,t}=\left[-\nabla f(y_{i,t})+\sum_{j=1,j\neq i}^N \alpha_{ij}g(y_{i,t}-y_{j,t})\right] \tilde{\gamma}dt+\tau\tilde{\gamma}dB_{i,t}, \label{opt eq: dy_it}
\end{equation}
where $\tilde{\gamma}=1/\tilde{\Delta} t$ and $\tau=\sigma\sqrt{{\tilde{\Gamma} \tilde{\Delta} t}}$.

In this paper, we characterize the performance of the flocking-based approach using the average solution $\bar{x}_{t}=({1}/{N})\sum\nolimits_{i=1}^{N}x_{i,t}$.
%\begin{remark}
%	{\color{blue} We characterize the performance of the flocking-based approach using the average solution $\bar{x}_{t}=({1}/{N})\sum\nolimits_{i=1}^{N}x_{i,t}$.}
%	\end{remark}

\begin{remark}
	In what follows, we shall use the same specification for $g(\cdot )$ as in \cite{gazi2004class}, i.e., $g(\cdot )$ is an odd function of
	the form: 
	\begin{equation}  \label{opt eq: g()}
	g(x)=-x[g_{a}(\Vert x\Vert )-g_{r}(\Vert x\Vert )],
	\end{equation}%
	where $g_{a}:\mathbb{R}^{+}\rightarrow \mathbb{R}^{+}$ represents (the
	magnitude of) the attraction term and it has long range, whereas $g_{r}:%
	\mathbb{R}^{+}\rightarrow \mathbb{R}^{+}$ represents (the magnitude of) the
	repulsion term and it has short range, and $\Vert \cdot \Vert $ is the
	Euclidean norm. 
	We assume that $J(\cdot)$ has a unique minimizer, and there is an equilibrium distance $\rho >0$ such that $g_{a}(\rho )=g_{r}(\rho )$, and for $\Vert x\Vert >\rho $ we have $g_{a}(\Vert x\Vert )>g_{r}(\Vert x\Vert )$, and for $\Vert x\Vert <\rho $ we have $g_{a}(\Vert x\Vert )<g_{r}(\Vert x\Vert )$.  
	%Further, we assume there is an equilibrium distance $\rho >0$
	%such that $g_{a}(\rho )=g_{r}(\rho )$, and for $\Vert y\Vert >\rho $ we have 
	%$g_{a}(\Vert x\Vert )>g_{r}(\Vert x\Vert )$, and for $\Vert x\Vert <\rho $
	%we have $g_{a}(\Vert x\Vert )<g_{r}(\Vert x\Vert )$. 
	In this work we	consider linear attraction functions, i.e., $g_{a}(\Vert x\Vert )=a$ for
	some $a>0$ and all $\Vert x\Vert $, and repulsion functions satisfying $% 
	g_{r}(\Vert x\Vert )\Vert x\Vert ^{2}\leq b$ uniformly for some $b>0$.
	
	The choice of parameters $a$ (i.e. attraction) and $b$ (i.e. repulsion) reflects the emphasis on exploration (higher values of $b$) versus exploitation (higher values of $a$).  The potential function is reminiscent of penalty function methods for constrained optimization in which the gradient of the objective function is perturbed so as to ensure updated solutions remain within the feasible region. 
	The difference is that in the flocking approach, potential-induced attraction/repulsion forces keep the updated solutions in a moving ball with fixed size rather than a rigid region. In light of its functionality, the analysis would not change much if we had adopted a different potential function.
	
	%{\color{blue} Jensen: we should comment on this choice. Would the analysis change too much if we had adopted a different potential ?
	%what is the relationship with penalty barrier functions used in optimization ? }
\end{remark}

\section{Analysis}
\label{opt sec:cohesion}

In this section we study the stochastic processes $\{y_{i,t}:t\ge 0\}$ associated with each one of the $N>1$ threads in the flocking-based approach. 
%{\color{blue}As mentioned in the last section},
The average solution $\bar{y}_{t}=({1}/{N})\sum\nolimits_{i=1}^{N}y_{i,t}$ will be of particular importance in characterizing the performance of the flocking-based approach. 
This part of the analysis pertains to a characterization of cohesiveness of the solutions identified by the different threads. To this end, we will analyze the process $\{\overline{V}_{t}:t>0\}$ defined as
\[
\overline{V}_t=\frac{1}{N}\sum_{i=1}^{N}\frac{1}{2}\Vert y_{i,t}-\bar{y}_{t} \Vert ^{2}.
\]

In the analysis, we will frequently make use of Ito's Lemma as stated below (see \cite{oksendal2003stochastic}).
	\begin{lemma}
		(Ito's lemma) Let 
		\[dX_t=udt+vdB_t\]
		be an $n$-dimensional Ito process . Let $g_{t,x}$ be a twice differentiable map from $[0,\infty)\times\mathbb{R}^n$ into $\mathbb{R}$. Then the process
		\[Y_t=g_{t,X_t}\]
		is again  an Ito process with
		\[dY_t=\frac{\partial g}{\partial t}dt+\sum_i \frac{\partial g}{\partial x_i}dX_i+\frac{1}{2}\sum_{i,j}\frac{\partial^2 g}{\partial x_i\partial x_j}dX_idX_j\]
		where $dB_idB_j=\delta_{ij}dt$  ($\delta_{ij}=1$ if $i=j$ and $\delta_{ij}=0$ otherwise), $dB_idt=dtdB_i=0$.
	\end{lemma}

\subsection{Preliminaries}

We consider the stochastic differential equation
governing $\overline{V}_t$. Let $V_{i,t}=({1}/{2})\Vert e_{i,t}\Vert ^{2}$ with $e_{i,t}=y_{i,t}-\bar{y}_{t}$. We have
$ \overline{V}_t=({1}/{N})\sum_{i=1}^{N}V_{i,t}$.
Applying Ito's lemma, 
\begin{equation*}
d{V_{i,t}}=d{e_{i,t}}\cdot e_{i,t}+\frac{1}{2}de_{i,t}\cdot de_{i,t},
\end{equation*}
where $d{e_{i,t}}=d{y}_{i,t}-d{\overline{y}_{t}}$. 

\begin{lemma}
	\label{opt lemma1}
	Suppose relation (\ref{opt eq: dy_it}) holds, and assume a linear attraction function $g_{a}(\Vert x\Vert )=a$. Then $\overline{V}_t$ satisfies
	\begin{align} \label{opt eq:doV}
	d\overline{V}_t 
	= & -\frac{a}{N}\sum_{i=1}^{N}\sum_{j=1,j\neq
		i}^{N}\alpha_{ij}(e_{i,t}-e_{j,t})^T e_{i,t}\tilde{\gamma}dt+ \frac{1}{2N}\sum_{i=1}^{N}\sum_{j=1,j\neq i}^{N}\alpha_{ij}g_{r}(\Vert
	y_{i,t}-y_{j,t}\Vert )\Vert y_{i,t}-y_{j,t} \Vert ^2 \tilde{\gamma}dt \notag\\
	& -\frac{1}{N}%
	\sum_{i=1}^N \nabla^T f(y_{i,t}) e_{i,t} \tilde{\gamma}dt + \frac{\tau}{N} \tilde{\gamma}\sum_{i=1}^N {dB_{i,t}}^T e_{i,t}+\frac{m\tau^2 \tilde{\gamma}^2(N-1)%
	}{2N}dt. 
	\end{align}
\end{lemma}

\begin{proof}
	See Appendix \ref{opt App:lemma1}. 
\end{proof}

\subsection{Cohesiveness}

Let $L=[l_{ij}]$ be	the Laplacian matrix associated with the adjacency matrix $A$.
%, where $l_{ii}=\sum_j \alpha_{ij}$ and $l_{ij}=-\alpha_{ij}$ where $i\neq j$. For an undirected graph, the Laplacian matrix is symmetric positive semi-definite (\cite{godsil2013algebraic}).
Notice that
\begin{align}
	\label{opt eq: sum1}
	-\sum_{i=1}^{N}\sum_{j=1,j\neq
		i}^{N}\alpha_{ij}(e_{i,t}-e_{j,t})^T e_{i,t}=\sum_{i=1}^{N}\sum_{j=1}^{N}l_{ij}(e_{i,t}-e_{j,t})^T e_{i,t}=-\sum_{i=1}^{N}\sum_{j=1}^{N}l_{ij}e_{j,t}^T e_{i,t}.
	\end{align}
Let $\me_t=[e_{1,t}^T,\ldots,e_{N,t}^T]^T$. Since graph $\mathcal{G}$ is connected, $\lambda_2(L)>0$ (see \cite{godsil2013algebraic}) and
\begin{align}
\label{opt eq: sum2}
\sum_{i=1}^{N}\sum_{j=1}^{N}l_{ij}e_{j,t}^T e_{i,t}=\me_t^T(L\otimes I_m)\me_t\ge \lambda_2\me_t^T \me_t=\lambda_2\sum_{i=1}^{N}\|e_{i,t}\|^2.
\end{align}
Here $\lambda_2:=\lambda_2(L)$ is the second-smallest eigenvalue of $L$, also called the {\em algebraic connectivity} of $\mathcal{G}$.

We make the following standing assumptions.
	\begin{assumption}
		\label{opt asp:gradient_bounded}
		(Bounded gradient) There exists $\eta >0$ such that $\Vert
		\nabla f(x)\Vert \leq \eta $ for all $x$.
	\end{assumption}
	\begin{assumption}
		\label{opt asp:gradient_strconvexity} {(Strong convexity) $(\nabla f (x) -\nabla
			f (x^{\prime}))^T (x-x^{\prime}) \geq \kappa \|x-x^{\prime}\|^{2}$ for some 
			$\kappa>0$ and for all $x,x^{\prime}$}.
	\end{assumption}
The following result provides a characterization of degree of cohesiveness of sample paths associated with different individual threads.

\begin{theorem}
	\label{opt thm1}
	Suppose relation (\ref{opt eq: dy_it}) holds, and assume a linear attraction function $g_{a}(\Vert x\Vert )=a$  and a repulsion function satisfying $% 
		g_{r}(\Vert x\Vert )\Vert x\Vert ^{2}\leq b$. Then the ensemble average of $%
		\overline{V}_t$ is uniformly bounded in $t$ under either Assumption \ref{opt asp:gradient_bounded} or Assumption \ref{opt asp:gradient_strconvexity}. In particular,
	\begin{enumerate}
		\item If Assumption \ref{opt asp:gradient_bounded} is satisfied, then
		\begin{align}
		\mathbb{E}[\overline{V}_t ] \le e^{-c_1\tilde{\gamma} t}\overline{V}_0+\frac{c_2}{c_1}%
		(1-e^{-c_1\tilde{\gamma} t}).
		\end{align}
		where $c_1\in \left(0,2a\lambda_2\right)$ is arbitrary, and
		\begin{align*}
		c_2=\frac{\eta^2}{2(2a\lambda_2-c_1)}+\frac{b | Tr(L) |}{2N}+\frac{m\tau^2\tilde{\gamma}(N-1)}{2N}.
		\end{align*}
		In particular,
		\begin{equation*}
		\mathbb{E}[\overline{V}_{t}] \leq \max \{\overline{V}_{0},\psi_1 ^{\ast
		}(N)\},
		\end{equation*}%
		and in the long run, $\mathbb{E}[\overline{V}_{t}] \leq \psi_1 ^{\ast }(N)$
		where
		\begin{equation*}
		\psi_1 ^{\ast }(N)=\frac{1}{2a\lambda_2}{\left[\sqrt{\frac{\eta^2}{4a\lambda_2}+\frac{b | Tr(L) |}{2N}+\frac{m\tau^2\tilde{\gamma}(N-1)}{2N}}+\sqrt{\frac{\eta^2}{4a\lambda_2}}\right]^2}.
		\end{equation*}
		
		\item If Assumption \ref{opt asp:gradient_strconvexity} is satisfied, then
		\begin{align*}
		%											\label{social neq:doV temp asp_con}
		\mathbb{E}[\oV_t] \le e^{-2(\kappa+a\lambda_2)\tilde{\gamma}t}\oV_0+\left[\frac{b | Tr(L) |}{4N(\kappa+a\lambda_2)}+\frac{m\tau^2\tilde{\gamma} (N-1)}{4N(\kappa+a\lambda_2)}\right]\left[1-e^{-2(\kappa+a\lambda_2)\tilde{\gamma}t}\right].
		\end{align*}
		In the long run,
		\begin{align*}
		%											\label{social neq:doV long asp_con}
		\mathbb{E}[\overline{V}_t] \le \psi_2^*(N) = \frac{b | Tr(L) |}{4N(\kappa+a\lambda_2)}+\frac{m\tau^2\tilde{\gamma} (N-1)}{4N(\kappa+a\lambda_2)}.
		\end{align*}
		
	\end{enumerate}
	
\end{theorem}

\begin{proof}
	See Appendix \ref{opt App:thm1}.
\end{proof}

\begin{remark}
	Note that the upper bound on the ensemble average of $\overline{V}_t$ is decreasing in $a$ (attraction potential)
	and increasing in $b$ (repulsive potential).  Hence, the relative strength of these parameters implies a trade-off between exploration (less cohesive solutions) and exploitation (more cohesive solutions).
\end{remark}

\begin{remark}
	Note further that the algebraic connectivity $\lambda_2$ is critical in determining the upper bound of $\mathbb{E}[\overline{V}_t]$. When $N$ is fixed, a larger $\lambda_2$ leads to a smaller upper bound.	In a complete graph, $\lambda_2$ achieves its maximum value $N$, and $Tr(L)=N(N-1)$. In this situation,
	
	\begin{align*}
	\lim_{N\rightarrow \infty }\psi_1 ^{\ast }(N)
	&=\lim_{N\rightarrow \infty }\frac{%
		1}{2aN}{\left[\sqrt{\frac{\eta^2}{4aN}+\frac{b(N-1)}{2}+\frac{m\tau^2(N-1)}{%
				2N}\tilde{\gamma}}+\sqrt{\frac{\eta^2}{4aN}}\right]^2}
	=\frac{b}{4a}.\\
	\lim_{N\rightarrow \infty }\psi_2 ^{\ast }(N)& =\lim_{N\rightarrow \infty }\left[\frac{b(N-1)}{4(\kappa+aN)}+\frac{m\tau^2\tilde{\gamma} (N-1)}{4N(\kappa+aN)}\right]=\frac{b}{4a}.
	\end{align*}%
	With a large number of threads, cohesiveness is
	ensured by the choice of ${b}/{a}$ governing the
	interplay between inter-individual attraction and repulsion.
\end{remark}
%{\color{blue} 
%	The bound $\psi^*(N)$ we obtain is tight when $N$ gets large. As $N\rightarrow\infty$, $\psi^*(N)$ approaches the exact mean of $\overline{V}_{t}$ in the long run.}

%	{\color{blue} This bound is not tight. Should we add discussion on it?}

%{\color{blue} Jensen: please add comment on the tightness of the upper bound}

\section{Noise Reduction in Convex Optimization}
\label{opt sec: increased convex}

In this section we formalize the  noise reduction properties of the flocking-based algorithmic scheme
for convex optimization. 
Repulsion amongst threads prevents duplication of search effort which may arise for instance, when there are multiple locally optimal solutions. Thus, for convex optimization problems, there is no need for a ``repulsion" amongst individual threads and in this section we set $g_r(\|x\|)=0$. As we shall see below, when the underlying problem is not convex, repulsion amongst threads does facilitate the identification of a globally optimal solution.

We introduce an additional regularity assumption as follows.
\begin{assumption}
	\label{opt asp:Lipschitz} {(Lipschitz) $\| \nabla f(x)-\nabla f(x^{\prime})\|\leq
		\mu\|x-x^{\prime}\|$ for some $\mu> 0$ and for all $x,x^{\prime}$}.
\end{assumption}
%{\bf {\color{red} Jensen: I think the assumption $\mu\geq \kappa$ may be considered too strong. 
%Couldn't we use $\|mu - \kappa|$ in obtaining a bound of $dW_t$ based upon bounds for $dU_t$ and $d\bar{V}_t$ ? \\Ans: From Assumption 1, we can get $\| \nabla f(x)-\nabla f(x^{\prime})\|\geq
%\kappa\|x-x^{\prime}\|,\forall x,x'$. So $\mu$ must be greater then $\kappa$. Since we can always increase $\mu$ and preserve the convergence rate, I feel that the existence of such $\mu$ is not a strong assumption. What do you think?}}
%{\bf {\color{blue} Jensen: I see. So for presentation purposes wouldn't it be better to state the Lipschitz assumption
%without the requirement $\mu \geq \kappa$ and then in the analysis introduce $\mu := \max\{\mu,\kappa\}$ so that
%$\mu \geq \kappa$ and $\| \nabla f(x)-\nabla f(x^{\prime})\|\leq \mu\|x-x^{\prime}\|$. Also,
%$\mu - \kappa = \max\{\mu-\kappa,0\}$.
%}}
Since $g_r(\|x\|)=0$, equation (\ref{opt eq:doV}) can be simplified to
\begin{multline}  
d\overline{V}_t 
=  -\frac{a}{N}\sum_{i=1}^{N}\sum_{j=1,j\neq
	i}^{N}\alpha_{ij}(e_{i,t}-e_{j,t})^T e_{i,t}\tilde{\gamma}dt-\frac{1}{N}%
\sum_{i=1}^N \nabla^T f(y_{i,t}) e_{i,t}\tilde{\gamma} dt + \frac{\tau}{N}\tilde{\gamma} \sum_{i=1}^N {dB_{i,t}}^T e_{i,t}\\
+\frac{m\tau^2\tilde{\gamma}^2(N-1)%
}{2N}dt.\label{opt eq:doV_simp}
\end{multline}

Let us introduce a measure $U_{t}={(1 /2)} \Vert \bar{y}_{t}-x^{\ast }\Vert ^{2}={(1 /2)} \Vert \bar{x}_{t/\tilde{\Gamma}}-x^{\ast }\Vert ^{2}$, of the distance between the average solution identified by all threads at time $t/\tilde{\Gamma}$ and the unique optimal solution $x^{\ast }$.
Let $F_{i,t}=({1}/{2})\Vert y_{i,t}-x^{\ast
}\Vert ^{2}$
and $\overline{F}_{t}=({1}/{N})\sum_{i=1}^{N}F_{i,t}$. 
Notice that 
\begin{align}  \label{opt eq: FUV}
\overline{F}_{t} &=\frac{1}{N}\sum_{i=1}^{N}\frac{1}{2}(y_{i,t}-\bar{y}_{t}+%
\bar{y}_{t}-x^{\ast })\cdot (y_{i,t}-\bar{y}_{t}+\bar{y}_{t}-x^{\ast }) 
\notag \\
&=\frac{1}{N}\sum_{i=1}^{N}\frac{1}{2}\Vert y_{i,t}-\bar{y}_{t}\Vert ^{2}+%
\frac{1}{N}\sum_{i=1}^{N}\frac{1}{2}\Vert \bar{y}_{t}-x^{\ast }\Vert ^{2}+\frac{1}{N}\sum_{i=1}^{N}(y_{i,t}-\bar{y}_{t})\cdot (\bar{y}_{t}-x^{\ast })\notag\\
&=%
\overline{V}_{t}+U_{t}.
\end{align}%

The following results provide a characterization of the performance of the flocking-based approach to stochastic optimization under Assumptions \ref{opt asp:gradient_strconvexity} and \ref{opt asp:Lipschitz}.

\begin{lemma}
	\label{opt thm2}
	Suppose relation (\ref{opt eq: dy_it}) and Assumptions \ref{opt asp:gradient_strconvexity} and \ref{opt asp:Lipschitz} hold. Assume a linear attraction function $g_{a}(\Vert x\Vert )=a$  and a repulsion function $% 
		g_{r}(\Vert x\Vert )=0$. Then the
	ensemble average of $U_{t}$ is uniformly bounded:
	\begin{align*}
	\label{opt neq: EU}
	\mathbb{E}[U_t]
	%= &{1 \over 2}\mathbb{E}[\Vert \bar{y}_{t}-x^{\ast }\Vert ^{2}]  \nonumber\\
	\leq & e^{-2\kappa\tilde{\gamma} t}\left[U_0+\frac{(\mu-\kappa)}{a\lambda_2}\overline{V}_0\right] +\frac{m\tau^2\tilde{\gamma}}{4\kappa N}\left[1+\frac{(\mu-\kappa)(N-1)}{a\lambda_2}%
	\right](1-e^{-2\kappa\tilde{\gamma} t}).
	\end{align*}
	In the long-run	the upper bound is: 
	\begin{equation*}
	\lim \sup_{t \rightarrow \infty} \mathbb{E}[U_t]
	%{1 \over 2}\mathbb{E}[\Vert \bar{y}_{t}-x^{\ast }\Vert ^{2}] 
	\leq \phi ^{\ast }(N)=\frac{m\tau^2\tilde{\gamma}}{4\kappa N}\left[1+\frac{(\mu-\kappa)(N-1)}{a\lambda_2}%
	\right].
	\end{equation*}
\end{lemma}

\begin{proof}
	See Appendix \ref{opt App:thm2}. 
\end{proof}

\begin{theorem}
	\label{opt cor}
	Suppose relation (\ref{opt eq: dy_it}) and Assumptions \ref{opt asp:gradient_strconvexity} and \ref{opt asp:Lipschitz} hold. Assume a linear attraction function $g_{a}(\Vert x\Vert )=a$  and a repulsion function $% 
		g_{r}(\Vert x\Vert )=0$. We have
	\begin{multline}
	\label{opt neq: EUx}
	\frac{1}{2}\mathbb{E}[\|\ox_t-x^*\|^2]
	\leq e^{-2\kappa\tilde{\gamma}\tilde{\Gamma} t}\left[U_0+\frac{(\mu-\kappa)}{a\lambda_2}\overline{V}_0\right] +\frac{m\tau^2\tilde{\gamma}}{4\kappa N}\left[1+\frac{(\mu-\kappa)(N-1)}{a\lambda_2}%
	\right](1-e^{-2\kappa\tilde{\gamma}\tilde{\Gamma}  t}).
	\end{multline}
	In the long-run	the upper bound is: 
	\begin{equation}
	\lim \sup_{t \rightarrow \infty} \frac{1}{2}\mathbb{E}[\|\ox_t-x^*\|^2]
	%{1 \over 2}\mathbb{E}[\Vert \bar{y}_{t}-x^{\ast }\Vert ^{2}] 
	\leq \phi ^{\ast }(N)=\frac{m\tau^2\tilde{\gamma}}{4\kappa N}\left[1+\frac{(\mu-\kappa)(N-1)}{a\lambda_2}%
	\right].
	\end{equation}
\end{theorem}
\begin{proof}
	Given that $(1/2)\mathbb{E}[\|\ox_t-x^*\|^2]=(1/2)\mathbb{E}[\|\oy_{\tilde{\Gamma}t}-x^*\|^2]=\mathbb{E}[U_{\tilde{\Gamma} t}]$, the above results follow immediately from Lemma \ref{opt thm2}.
\end{proof}

\begin{remark}
	The convergence rate of $(1/2)\mathbb{E}[\|\ox_t-x^*\|^2]$ is characterized by factor $2\kappa\tilde{\gamma}\tilde{\Gamma}=2\kappa\tilde{\Gamma}/\tilde{\Delta} t$, i.e., stronger convexity of $f(\cdot)$, larger step sizes, and shorter sampling times accelerate convergence.
\end{remark}

Note that if we choose $a$ such that $a\lambda_2\sim N$, $\phi ^{\ast }(N) \sim {1}/{N}$, then the long-run upper bound of $\mathbb{E}[\|\ox_t-x^*\|^2]$ is monotonically decreasing in $N$. In what follows, we will show that the flocking-based approach exhibits a noise reduction property that is similar to that of a stochastic gradient algorithm based upon the average of $N$ gradient samples.

Let us assume there is no time overhead in the centralized algorithm so that $\Delta t=\tilde{\Delta} t$. 
%Without loss of generality, let $\Delta t=\tilde{\Delta} t=1$.
Also assume $\Gamma=\tilde{\Gamma}$. It follows that $\tilde{\gamma}=\gamma$
%$\tau_N =\sigma\sqrt{\Gamma/(\gamma N)}$, 
and $\tau = \sigma\sqrt{\Gamma/\gamma}$. 

The stochastic differential equation for the algorithm based upon the average of $N$ gradient samples is:
\begin{equation*}
dy_{t}=-\nabla f(y_{t})\gamma dt+\tau_N \gamma dB_{t}.
\end{equation*}
Let $G_{t}=({1}/{2})\Vert y_{t}-x^{\ast }\Vert ^{2}$. It follows that 
\begin{align}
\label{opt eq: G_t}
dG_{t}  =dy_{t}\cdot y_{t}+\frac{1}{2}dy_{t}\cdot dy_{t} =-\nabla
^{T}f(y_{t})y_{t}\gamma dt+\tau_N\gamma{dB_{t}}^{T}y_{t}+\frac{1}{2}m\tau_N^2\gamma^2dt.
\end{align}
Then,
\begin{align*}
dG_{t}  \le -2\kappa\gamma G_{t} dt+\frac{1}{2}m\tau_N^2\gamma^2dt+\tau_N\gamma{dB_{t}}^{T}y_{t}.
\end{align*}%
As in the proof of Lemma \ref{opt thm2}, it can be shown that 
\begin{align*}
\mathbb{E}[G_t]  \le e^{-2\kappa\gamma t}G_{0}+\frac{m\tau_N^2 \gamma}{4\kappa}(1-e^{-2\kappa\gamma t})=e^{-2\kappa\gamma t}G_{0}+\frac{m\sigma^2 \Gamma}{4\kappa N }(1-e^{-2\kappa\gamma t}).
\end{align*}
Therefore,
\begin{align*}
\frac{1}{2}\mathbb{E}{[\|x_t-x^*\|^2]}=\mathbb{E}[G_{\Gamma t}]  \le e^{-2\kappa\gamma\Gamma t}G_{0}+\frac{m\sigma^2 \Gamma}{4\kappa N }(1-e^{-2\kappa\gamma\Gamma t}).
\end{align*}
%Noticing that $2\kappa\gamma=2\kappa\tilde{\gamma}$, the convergence speeds of the two algorithms are identical. 
In the long run, 
\begin{align*}
\lim \sup_{t \rightarrow \infty} \frac{1}{2}\mathbb{E}{[\|x_t-x^*\|^2]}\le \frac{m \sigma^2 \Gamma}{4\kappa N }.
\end{align*}
Since $\tau =\sigma\sqrt{\Gamma/\gamma}$, a comparison with the upper bound obtained in Theorem \ref{opt cor} (when $a\lambda_2\sim N$) readily indicates that the flocking-based approach exhibits a noise reduction property that is similar to that of a stochastic gradient algorithm based upon the average of $N$ gradient samples.
In the next section, we show that the flocking-based approach outperforms (in real-time) stochastic gradient algorithm based upon the average of $N$ gradient samples when overhead due to synchronization is taken into account.

%{\color{blue} {\bf Remark 4:}
%	$\phi^*$ is exact when $\kappa=\mu$. When $\kappa< \mu$, tightness of $\phi^*$ depends on detailed information of $\nabla\sigma(\cdot)$. Convergence speed of $E[ U_t] $ is determined by  $2\kappa\tilde{\gamma}=2\kappa\tilde{\Gamma}/\tilde{\Delta} t$, i.e., a larger gradient/step size and a smaller sample time leads to faster convergence.}

%{\color{blue} Markov's inequality for finite-time performance}

\subsection{Real-time Performance Comparison}
\label{opt subsec: real}
When sampling is undertaken in parallel, synchronization is needed to execute the tasks that can not be executed in parallel.
Hence, the improvement obtained by parallel sampling and {\em centralized} gradient estimation is limited by overhead related to {\em (i)} time spent gathering samples and {\em (ii)} synchronization\footnote{See for example \cite{hill2008amdahl} for a discussion on Amdahl's law in multi-core processing.}.
In what follows we will account for overhead by assuming the total time needed to implement an iteration of stochastic gradient algorithm with $N$ samples obtained in parallel is monotonically increasing in $N$, i.e., $\Delta t \sim N^{{1}/{\beta}}\tilde{\Delta} t$ where $\beta>1$.  
%The sampling rate is therefore $C_N ={N}/{\Delta t} \sim N^{1-{1}/{\beta}}$. 
The parameter $\beta>1$ encapsulates the relative burden of overhead so that when $\beta \gg 1$, the burden is relatively weak but increases with values closer to $1$.

In a way similar to the analysis presented in the previous section, we get
\begin{align*}
dG_t \geq -2\mu\gamma G_{t} dt+\frac{1}{2}m\tau _{N}^{2}\gamma^2 dt+{\tau _{N}}\gamma{dB_{t}}%
^{T}y_{t}.
\end{align*}%
A lower bound could then be obtained:
\begin{equation}
\label{opt neq: EG lower}
\mathbb{E}[G_t]  \geq e^{-2\mu\gamma t}G_{0}+\frac{m\tau _{N}^{2}\gamma}{4\mu}%
(1-e^{-2\mu\gamma t}).
\end{equation}%
Then,
\begin{align}
\label{opt neq: EG lowerx}
\frac{1}{2}\mathbb{E}{[\|x_t-x^*\|^2]}=\mathbb{E}[G_{\Gamma t}]  \geq e^{-2\mu\gamma\Gamma t}G_{0}+\frac{m\tau _{N}^{2}\gamma}{4\mu}%
(1-e^{-2\mu\gamma\Gamma t}),
\end{align}
and in the long run, 
\[\frac{1}{2}\mathbb{E}{[\|x_t-x^*\|^2]} \geq \frac{{m\tau _{N}^{2}}\gamma}{4\mu }=\frac{m\sigma^{2}\Gamma}{4\mu N}. \]
In what follows we consider two specific scenarios to compare the two algorithms in both convergence rate and ultimate error bound.

In the case that the two algorithms use the same step size $\Gamma=\tilde{\Gamma}$, 
recalling that $\gamma=1/\Delta t$ and $\tilde{\gamma}=1/\tilde{\Delta} t$, we have $\gamma=(\tilde{\Delta}t/\Delta t)\tilde{\gamma}\sim N^{-1/\beta}\tilde{\gamma}$. 
By (\ref{opt neq: EG lowerx}), the convergence rate of $(1/2)\mathbb{E}{[\|x_t-x^*\|^2]}$ is characterized by factor $2\mu\gamma\Gamma$, as compared to $2\kappa \tilde{\gamma}\tilde{\Gamma}$ under the flocking-based algorithm.
	Since $2\mu\gamma\Gamma/(2\kappa \tilde{\gamma}\tilde{\Gamma})\sim N^{-1/\beta}$,
	%In light of the fact that $x_t=y_{\Gamma t}$ and $x_{i,t}=y_{i,{\tilde{\Gamma} t}}$, 
	the convergence rate of the centralized scheme is slower than that of a flocking-based scheme. The long-run performance of the centralized implementation is bounded below by
$
{m\sigma^{2}\Gamma}/({4\mu N}) \sim {1}/{N},
$
which is on the same level of the flocking-based scheme.

In the case that the two algorithms use stepsizes proportional to the average sampling times, i.e., $\Gamma/\Delta t=\tilde{\Gamma}/\tilde{\Delta} t$,
%In the case that the two algorithms run in the same time scale, i.e., $\gamma=\tilde{\gamma}$ (refer to (\ref{opt eq:dy_t centralized}) and (\ref{opt eq: dy_it})), 
$(1/2)\mathbb{E}[\|\bar{x}_t-x^{\ast}\|^2]$ and $(1/2)\mathbb{E}[\|x_t-x^{\ast}\|^2]$
have comparable convergence rates since 
\[2\mu\gamma\Gamma=2\mu \Gamma/\Delta t=2\mu \tilde{\Gamma}/\tilde{\Delta} t=2\mu\tilde{\gamma}\tilde{\Gamma}.\]
%the two algorithms have the same convergence speed since $2\mu\gamma=2\mu\tilde{\gamma}$. 
However, $(1/2)\mathbb{E}[\|x_t-x^{\ast}\|^2]$ has a long-run lower bound
\[\frac{m\tau _{N}^{2}\gamma}{4\mu}=\frac{m\sigma^{2}\Gamma}{4\mu N}=\frac{m\sigma^{2}\tilde{\Gamma}}{4\mu N}\frac{\Delta t}{\tilde{\Delta}t}\sim N^{1/\beta-1}.\]
We can formalize the claim that the flocking approach (each thread has a single gradient sample) outperforms the stochastic gradient approach based upon $N$ samples per step in finite time.

\begin{proposition}
	\label{opt prop 1}
	Suppose relation (\ref{opt eq:dy_t centralized}), (\ref{opt eq: dy_it}) and Assumptions \ref{opt asp:gradient_strconvexity}, \ref{opt asp:Lipschitz} hold. Assume a linear attraction function $g_{a}(\Vert x\Vert )=a$  and a repulsion function $% 
		g_{r}(\Vert x\Vert )=0$.
	Also assume $\Delta t=N^{1/\beta}\tilde{\Delta}t$ for all $N > 1$ and for some $\beta>1$, and $\Gamma/\Delta t=\tilde{\Gamma}/\tilde{\Delta} t$. Then there exists $N^*<\infty$ and $t^* < \infty$ such
	\begin{equation*}
	\frac{1}{2}\mathbb{E}[\Vert x_{t}-x^{\ast }\Vert ^{2}] > \frac{1}{2}\mathbb{E}[\Vert \bar{x}_{t}-x^{\ast }\Vert ^{2}] 
	\end{equation*}
	for all $N>N^*$ and $t>t^*$. 
\end{proposition}
\begin{proof}
	See Appendix \ref{opt App:prop1}.
\end{proof}

\begin{remark}
	It is of interest to analyze algorithm (\ref{opt eq: x(i,k)}) without adding the flocking term 
	\[-\sum_{j=1,j\neq
			i}^{N}\alpha_{ij}\nabla_{x(i,k)} J(\|x(i,k)-x(j_i,k)\|).\] In this case the dynamics of different threads are independent:
		\begin{align*}
		dy_{i,t}=-\nabla f(y_{i,t})\tilde{\gamma}dt+\tau\tilde{\gamma}dB_{i,t}.
		\end{align*}
		%	The solutions identified by each thread are independent.
		Let $H_{i,t}=(1/2)\|y_{i,t}-x^*\|^2$. It can be shown that
		\begin{equation*}
		\mathbb{E}[H_{i,t}]  \geq e^{-2\mu\tilde{\gamma} t}H_{i,0}+\frac{m\tau^{2}\tilde{\gamma}}{4\mu}%
		(1-e^{-2\mu\tilde{\gamma} t}).
		\end{equation*}%
		Hence
		\begin{equation*}
		\frac{1}{2}\mathbb{E}[\|x_{i,t}-x^*\|^2]=\mathbb{E}[H_{i,\tilde{\Gamma}t}]  \geq e^{-2\mu\tilde{\gamma}\tilde{\Gamma} t}H_{i,0}+\frac{m\tau^{2}\tilde{\gamma}}{4\mu}%
		(1-e^{-2\mu\tilde{\gamma}\tilde{\Gamma} t}).
		\end{equation*}%
		The long-run lower bound ${m\tau^{2}\tilde{\gamma}}/{4\mu}$ is greater than the upper bound for a flocking-based algorithm with large $N$. Therefore, the performance of each individual thread without the flocking term is worse than that of a flocking one.
		
		With respect to the average solution $\bar{x}_t$, notice that $x_{i,t}$'s converge to i.i.d. limiting random variables $x_{i,\infty}$'s. Therefore as $N\rightarrow\infty$, $\bar{x}_t$ converges almost surely to the expectation of $x_{i,\infty}$ (same for all $i$'s) by the law of large numbers. 
		When the function $f(\cdot)$ is not symmetric with respect to $x^*$, we have that $(1/2)\mathbb{E}[\|\bar{x}_t-x^*\|^2]$ is not decreasing to $0$ as $N$ increases.
		%However, $\mathbb{E}[x_{i,\infty}]$ is in general not equal to $x^*$ when the function $f(\cdot)$ is not symmetric with respect to $x^*$ so that $(1/2)\mathbb{E}[\|\bar{x}_t-x^*\|^2]$ will not decrease to $0$ as $N$ increases. 
		Hence, without adding the flocking term $-\sum_{j=1,j\neq i}^{N}\alpha_{ij}\nabla_{x(i,k)} J(\|x(i,k)-x(j_i,k)\|)$, the performance of the algorithm is unsatisfactory.
\end{remark}

\section{Application to Non-Convex Optimization}
\label{opt sec: nonconvex}
In this section, we apply the flocking-based algorithm to the optimization of general non-convex functions. We will provide some motivating simulation examples first and then discuss the global asymptotic properties of this scheme.

\subsection{Simulation Examples}
In this part, we illustrate with a limited simulation testbed the performance benefits of a flocking-based approach when the objective function is not convex. 
The results indicate that the noise reduction property is maintained. They also suggest that a flocking-based approach seems better suited to escape locally optimal solutions than a stochastic gradient descent based upon the average of $N$ samples. This is likely due to the repulsive force which enforces a certain level of diversity in the set of candidate solutions. The flocking based-gradient descent dynamics are thus more  likely to lead to globally optimal solutions.
We assume that $\Gamma/\Delta t=\tilde{\Gamma}/\tilde{\Delta} t$ in the following simulation examples.

\subsubsection{Ackley's Function (Case 1)}	
\label{opt subsec: Ackley repul}
We first consider Ackley's function 
\begin{align}  \label{opt Ackley}
f(x,y) =  -20\exp\left(-0.2\sqrt{0.5\left(x^{2}+y^{2}\right)}%
\right) -\exp\left(0.5\left(\cos\left(2\pi x\right)+\cos\left(2\pi
y\right)\right)\right) + e + 20.
\end{align}
It has a global minimum at $x^{\ast}=(0,0)$ and various local optima.

We use $N=20$ parallel threads, and the attraction/repulsion function $%
g(x)=-x[4-800\exp(-\|x\|^2)]$ is adopted.
In the flocking-based approach, each thread is randomly connected with 8 other threads. Overhead parameter $\beta=5$. Sampling times are constant with $\tilde{\Delta} t=0.01$ and $\Delta t=N^{{1}/{\beta}}\tilde{\Delta} t\simeq 0.018$, respectively. Step sizes are $\tilde{\Gamma}=0.01$ and $\Gamma=0.018$. Noise level is $\sigma=5$.
Initially, all the sampling points were distributed randomly in the $[10,15]\times
[10,15]$ interval.  Simulations run 10 times (60s each). Performances of the centralized algorithm are shown in Figure \ref{opt fig:Ackley_repulsion}(a). It is clear that all independent threads got trapped in local optima. By contrast, we can see in Figure \ref{opt fig:Ackley_repulsion} that $\ox_t$ in the flocking-based scheme approaches the global optimum successfully. This is likely due to the repulsive force which enforces a certain level of diversity in the set of candidate solutions.
\begin{figure}
	\centering
	%	\subfigure[Sample Paths with
	%	Flocking discipline.]{\includegraphics[width=3in]{Ackley_repulsionD05P0N20.eps}} 
	%	\subfigure[Independent Sample
	%	Paths.]{\includegraphics[width=3in]{Ackley_repulsionD05P1N20.eps}} 
	\subfigure[Distance to optimum for sample average scheme with $N$ samples.]{\includegraphics[width=3in]{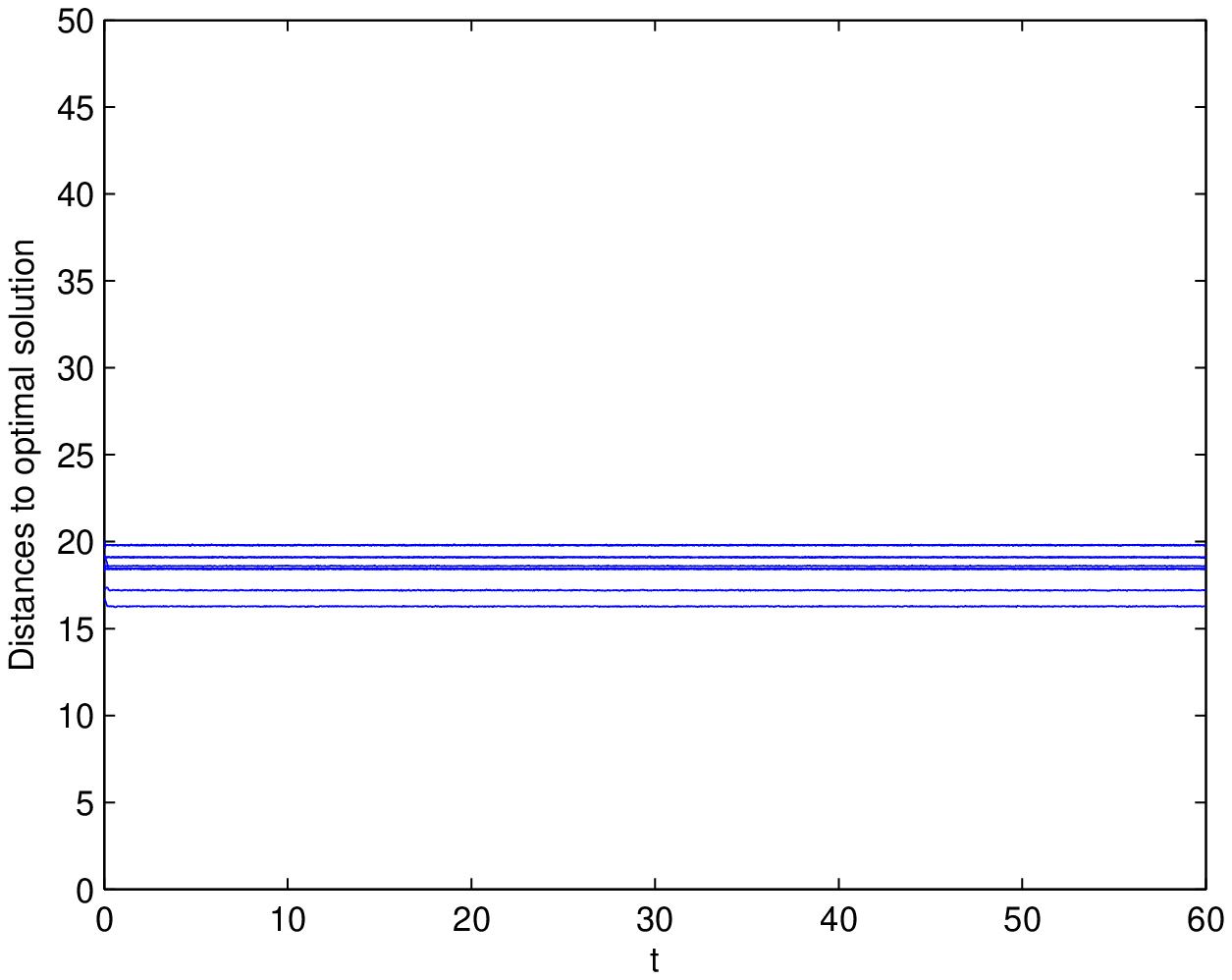}} 
	\subfigure[Distance
	to optimum with
	flocking discipline.]{\includegraphics[width=3in]{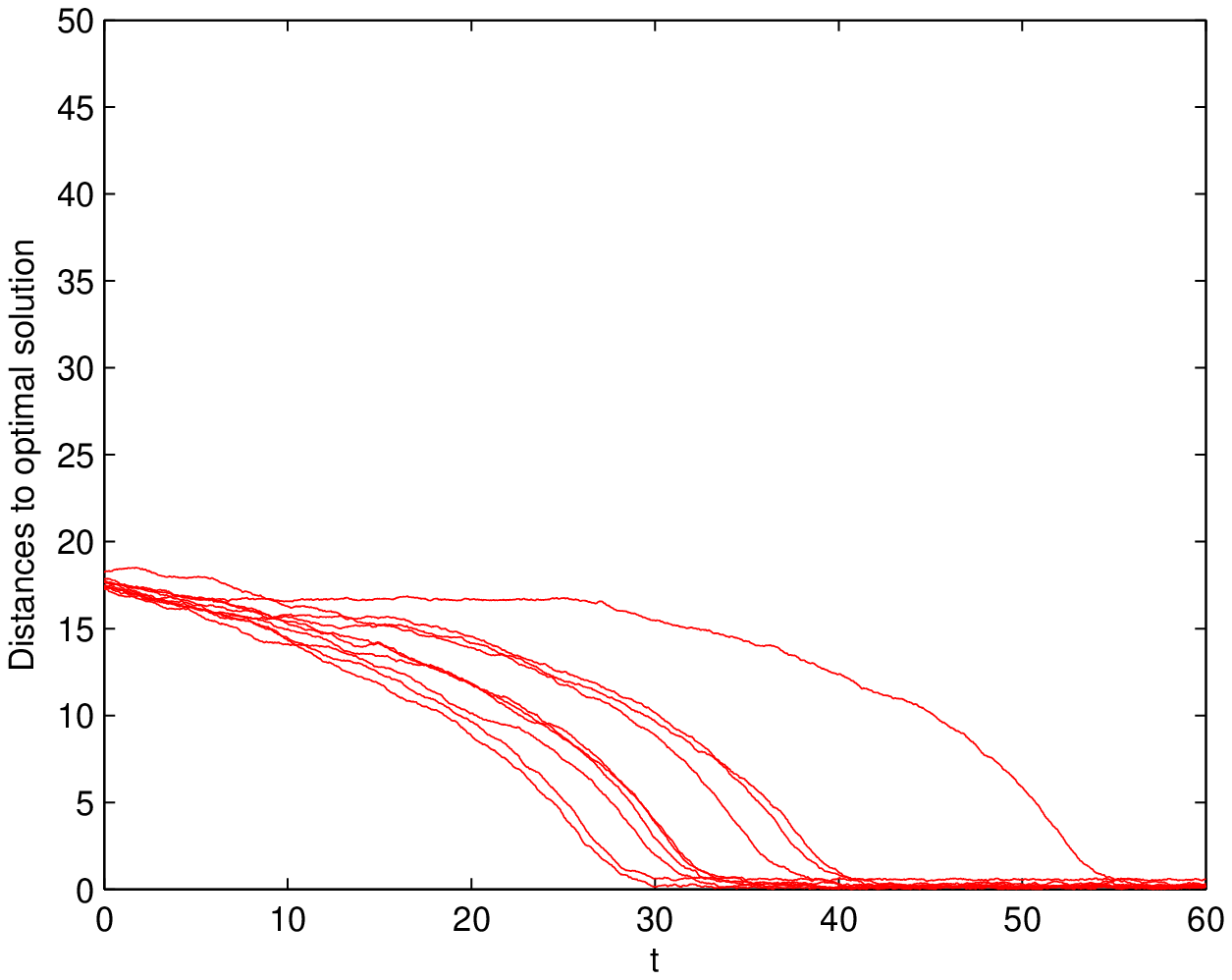}}
	\caption{Performance comparison for Ackley's function (Case 1).}
	\label{opt fig:Ackley_repulsion}
\end{figure}

\subsubsection{Ackley's Function (Case 2)}
\label{opt subsec: Ackley}

\begin{figure}
	\centering
	%	\subfigure[Sample Paths with
	%	Flocking discipline.]{\includegraphics[width=3in]{AckleyD07P1N30.eps}} 
	%	\subfigure[Independent Sample
	%	Paths.]{\includegraphics[width=3in]{AckleyD07P0N30.eps}} 
	\subfigure[Distance to optimum for sample average scheme with $N$ samples.]{\includegraphics[width=3in]{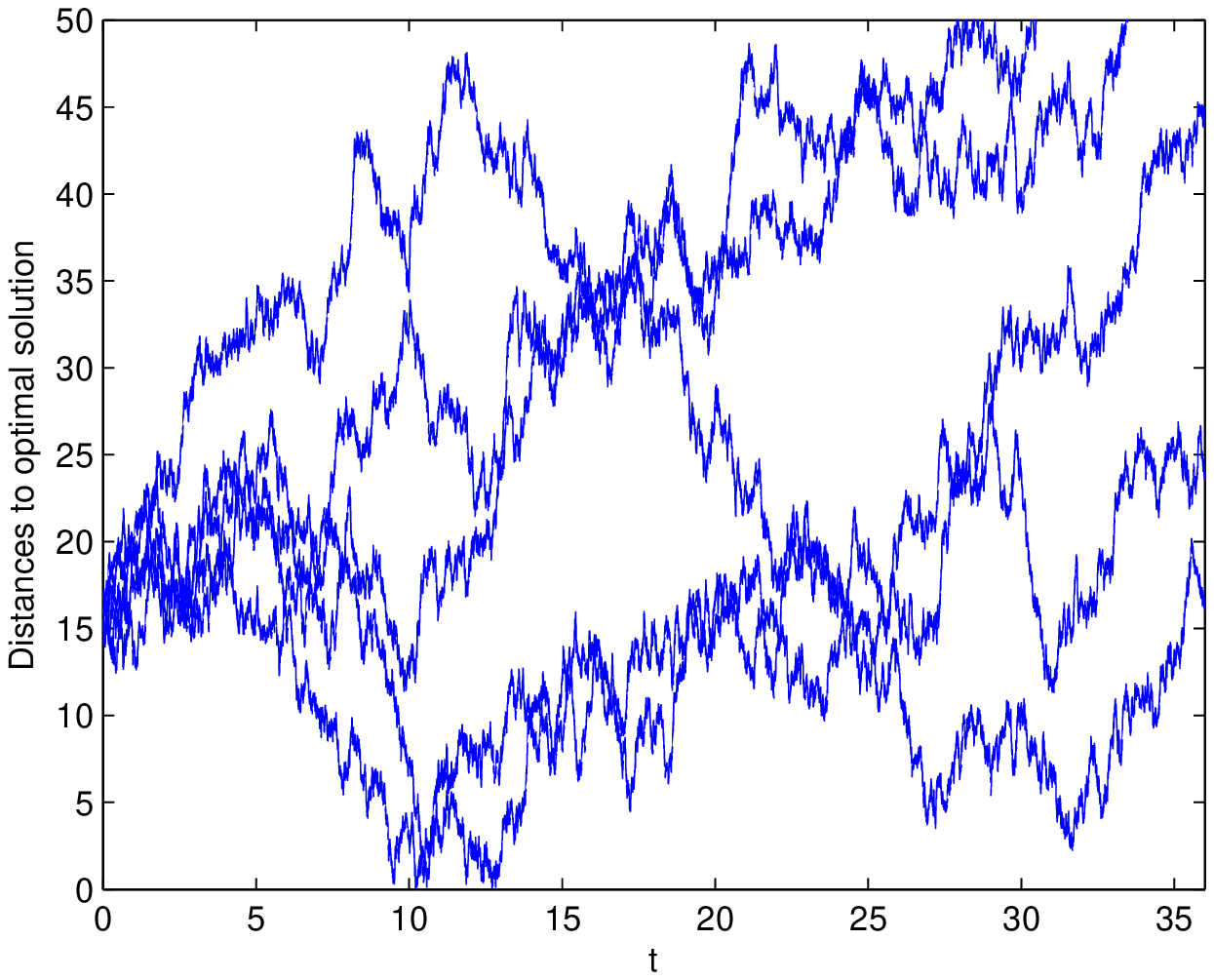}}
	\subfigure[Distance
	to optimum with
	flocking discipline.]{\includegraphics[width=3in]{{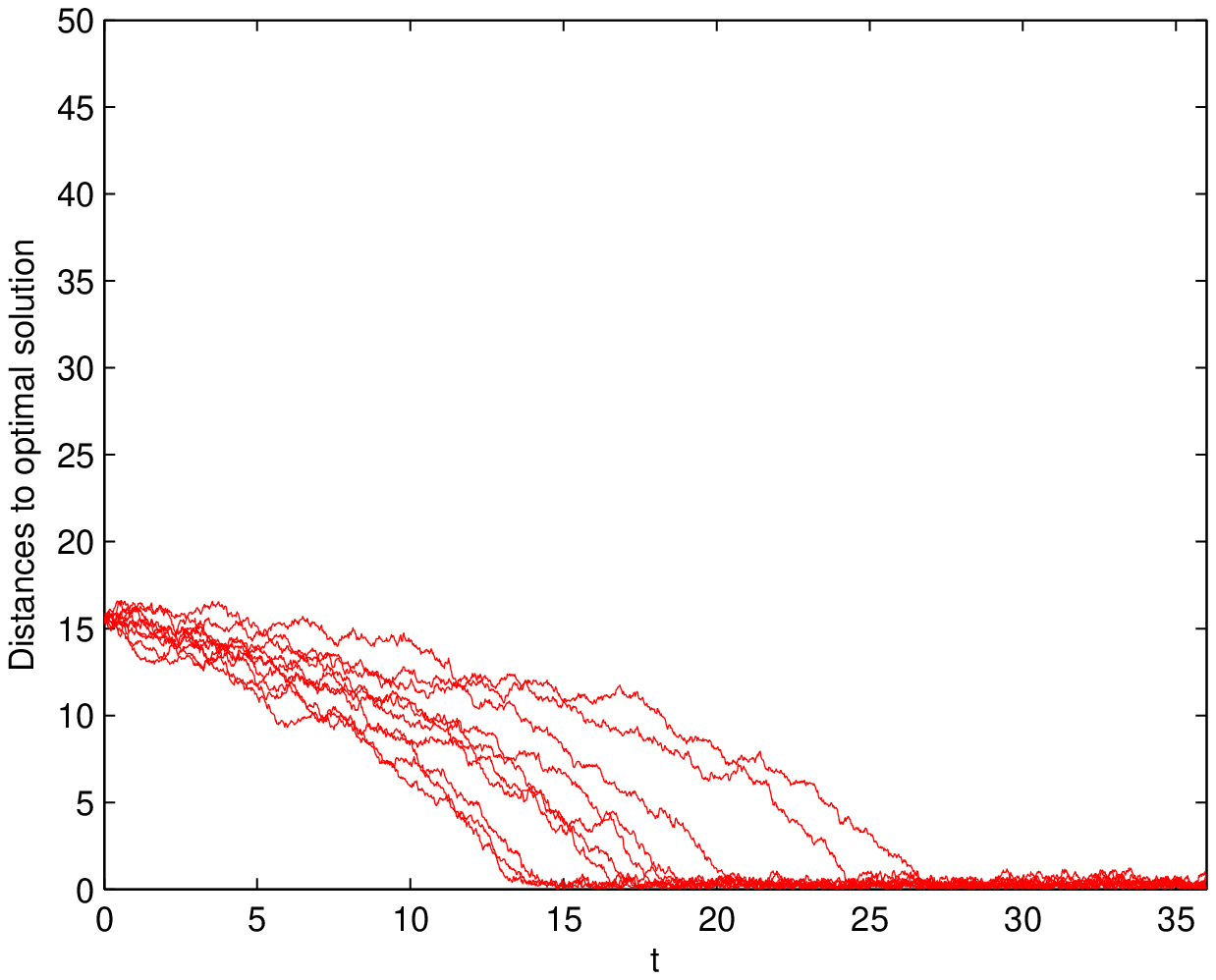}}} 
	\caption{Performance comparison for Ackley's function (Case 2).}
	\label{opt fig:Ackley}
\end{figure}
%We consider the same function as in Section \ref{opt subsec: Ackley repul}. 
In this case we assume that the noise level is $\sigma=35$.
We use 
$N=30$ computing threads, and the attraction/repulsion function $%
g(x)=-x[3-0.01\exp(-\|x\|^2)]$ is adopted. In the flocking-based approach, each thread randomly communicates with 8 other threads.
Overhead parameter $\beta=1.5$. Sampling times are constant with $\tilde{\Delta} t=0.04$ and $\Delta t=N^{{1}/{\beta}}\tilde{\Delta} t\simeq 0.184$, respectively. Step sizes are $\tilde{\Gamma}=0.04$ and $\Gamma=0.184$.
Initially, the sampling points were distributed randomly in the $[10,12]\times
[10,12]$ interval.  Simulations run for 36s.

We can see in Figure \ref{opt fig:Ackley}(a) that individual threads operating in
parallel are not able to approach the globally optimal solution. On the contrary, the flocking
discipline allows convergence to the global optimum successfully (see Figure \ref{opt fig:Ackley}(b)).

\subsection{Asymptotic Noise Reduction}
We now discuss the asymptotic noise reduction properties of the flocking-based algorithmic scheme for non-convex optimization. We start by reviewing the asymptotic performance of the centralized algorithm based upon the average of $N$ samples per step. Recalling (\ref{opt eq:dy_t centralized}),
\begin{equation*}
dy_{t}=-\nabla f(y_{t})\gamma dt+\tau_N \gamma dB_{t},  
\end{equation*}
The limiting density of $y_t$ (which solves a related Fokker-Planck equation  (see \cite{gardiner1985handbook,risken1984fokker})) is
\begin{align*}
\hat{\pi} (y_t)=\frac{\exp\{-2 f(x)/(\tau_N^2 \gamma)\}}{\int  \exp\{-2 f(y)/(\tau_N^2 \gamma)\}}=\frac{\exp\{-2N f(y)/(\sigma^2\tilde{\Gamma})\}}{\int  \exp\{-2N f(y)/(\sigma^2\tilde{\Gamma})\}}
\end{align*}
assuming $\Gamma=\tilde{\Gamma}$.

We return now to the flocking scheme. By equation (\ref{opt eq: dy_it}) and the relation $g(x)=-\nabla_x J(\|x\|)$, we have
\begin{align}
\label{opt eq: dx_it nonconvex}
dy_{i,t}=\left[-\nabla f(y_{i,t})-\sum_{j=1,j\neq i}^N \alpha_{ij}\nabla_{y_{i,t}}J\|y_{i,t}-y_{j,t}\|\right]\tilde{\gamma}dt+\tau\tilde{\gamma} dB_{i,t}.
\end{align}
Let $\my_t= [y_{1,t}^T,\ldots,y_{N,t}^T]^T \in \mathbb{R}^{N\times m}$, $\mathbf{B_t}=[B_{1,t}^T,\ldots,B_{N,t}^T]^T \in \mathbb{R}^{N\times m}$. Define
\begin{align}
H(\my_t)=\sum_{i=1}^{N}f(y_{i,t})+\frac{1}{2}\sum_{i=1}^{N}\sum_{j=1,j\neq i}^N \alpha_{ij} J(\|y_{i,t}-y_{j,t}\|).
\end{align}
We can rewrite (\ref{opt eq: dx_it nonconvex}) in a compact form:
\begin{align}
\label{opt eq: dx_t}
d\my_t=-\nabla H(\my_t)\tilde{\gamma}dt+\tau \tilde{\gamma}d\mathbf{B}_t.
\end{align}
The Fokker-Planck equation related to the stochastic differential equation (\ref{opt eq: dx_t}) is:
\begin{align}
\label{opt eq: Fokker}
\frac{\partial p_t}{\partial t}=-\nabla\cdot (\nabla H(\my)\tilde{\gamma}p_t)+\frac{\tau^2\tilde{\gamma}^2}{2}\nabla^2 p_t
\end{align}
where $p_t:=p(\my_t,t\mid \my_0,0)$ is the probability density of $\my_t$.

\begin{proposition}
	\label{opt prop 2}
	Suppose relation (\ref{opt eq: dy_it}) holds, and $\int_{}\exp\{-2H(\my)/\tau^2\}d\my$ is finite. Then $\my_t$ weakly approaches a unique equilibrium, which is a Gibbs distribution with density
	\begin{align}
	\label{opt eq: Gibbs}
	\pi(\my)=\frac{1}{K}\exp\{-2H(\my)/(\tau^2\tilde{\gamma})\},
	\end{align}
	where
	\begin{align*}
	K=\int \exp\{-2H(\my)/(\tau^2\tilde{\gamma})\}d\my.
	\end{align*}
\end{proposition}
Proposition \ref{opt prop 2} comes from the standard theory of diffusion (see \cite{gardiner1985handbook,risken1984fokker}). The assumption is required for the distribution (\ref{opt eq: Gibbs}) to be well-defined. It is satisfied when $H(\my)$ grows rapidly enough, or when the feasible region $X$ has reflected boundaries.

We now show that for large values of the parameter $a$ (attractive force) the asymptotic noise reduction properties of the flocking-based algorithmic scheme is the same to that of a centralized algorithm based upon the average of $N$ samples per step. 
\begin{theorem}
	\label{opt Thm: nonconvex}
	Suppose relation (\ref{opt eq: dy_it}) holds. Assume linear attraction functions $g_{a}(\Vert x\Vert )=a$ for some $a>0$ and repulsion functions satisfying $g_{r}(\Vert x\Vert )\Vert x\Vert ^{2}\leq b$.
	Then the asymptotic probability distribution of $\oy_t$ has density:
	\begin{align*}
	\pi^* (\oy)=\frac{\exp\left\{-2Nf(\oy)/(\sigma^2\tilde{\Gamma})\right\}}{\int \exp\left\{-2Nf(\oy)/(\sigma^2\tilde{\Gamma})\right\}d\oy} ,
	\end{align*}
	as $a\rightarrow \infty$.
\end{theorem}
\begin{proof}
	See Appendix \ref{opt App:thm3}.
\end{proof}
\begin{remark}
	The asymptotic probability distribution of $\ox_t$ is the same as that of $\oy_t$.
	% as $a\rightarrow \infty$.
\end{remark}
\begin{remark} The limiting probability density of $\ox_t$ as $a\rightarrow \infty$ does not depend on the specific network topology as long as $\lambda_2>0$, i.e., the network is connected. This is a very mild networking requirement satisfied by many simple topologies (e.g. ring, line, bus, mesh).
\end{remark}

\begin{remark}
	As $N\rightarrow \infty$, $\pi^*(\ox)$ concentrates on the global minima of $f(\cdot)$, in which case the average solution of the flocking scheme is guaranteed to approximate a global minimum (see \cite{geman1986diffusions} for a reference).
\end{remark}

%{\bf \color{red} here we should comment on $N \rightarrow \infty$.}

\section{Conclusions}
\label{opt sec: conclusion}

In recent years, the paradigm of cloud computing has emerged as an architecture for computing that makes use of distributed (networked) computing resources. In this paper, we analyze a distributed computing algorithmic scheme for stochastic optimization which relies on modest communication requirements amongst processors and most importantly, does not require synchronization. The proposed distributed algorithmic framework may provide significant speed-up in application domains in which sampling times are non-negligible. This is the case, for example, in the optimization of complex systems for which performance may only be evaluated via computationally intensive Òblack-boxÓ simulation models.

The scheme considered in this paper has $N>1$ computing threads operating under a connected network. At each step, each thread independently computes a new
solution by using a noisy estimation of the gradient, which is further perturbed by a
combination of repulsive and attractive terms depending upon
the relative distance to solutions identified by neighboring threads.
When the objective function is convex, we showed that a flocking-like approach for distributed stochastic optimization provides a noise reduction effect similar to that of a centralized stochastic gradient algorithm based upon the average of $N$ gradient samples at each step. When the overhead related to the time needed to gather $N$ samples and synchronization is not negligible, the flocking implementation outperforms a centralized stochastic gradient algorithm based upon the average of $N$ gradient samples at each step.
When the objective function is not convex, the flocking-based approach seems better suited to escape locally optimal solutions due to the repulsive force which enforces a certain level of diversity in the set of candidate solutions. Here again, we showed that the noise reduction effect is similar to that associated to the centralized stochastic gradient algorithm based upon the average of $N$ gradient samples at each step.

%\newpage
%\newpage

\section{Appendix}
\subsection{Notations}

\begin{table}[!h]	
	\centering
	\caption{Notations}
	\begin{tabular}{cl}
		Symbol & Meaning \\
		\midrule
		$m$ & Dimension of the solution space \\
		$N$          & Number of computing threads \\
		$\Gamma$ (respectively, $\tilde{\Gamma}$) & Step size for the centralized algorithm (respectively, flocking-based algorithm)\\
		$\Delta t$ & Average sampling time to gather $N$ samples in parallel  \\
		$\tilde{\Delta} t$ & Average sampling time to gather one sample\\
		$\varepsilon(k)$& Simulation noise from one sample (centralized algorithm)\\
		$\varepsilon(i,k)$ & Simulation noise from one sample (flocking-based algorithm)\\
		$\sigma$ & Variance of $\varepsilon(k)$ and $\varepsilon(i,k)$ (in each dimension)\\
		$\gamma$ & $1/\Delta t$ \\
		$\tilde{\gamma}$ & $1/\tilde{\Delta} t$ \\
		$\tau_N$ & $\sigma\sqrt{\Gamma\Delta t/N}$\\
		$\tau$ & $\sigma\sqrt{\tilde{\Gamma}\tilde{\Delta} t}$\\
		$a$ & Parameter for linear attraction in the flocking term\\
		$b$ & Bound on repulsion in the flocking term\\
		$\mathcal{G}$ & Interaction graph of ``flocking'' threads\\
		$\alpha_{ij}$ & Indicator of connectivity between thread $i$ and $j$\\
		$A$ & Matrix $[\alpha_{ij}]$ \\
		$L$ & Laplacian matrix of $A$\\
		$\lambda_2$ & Second-smallest eigenvalue of $L$ (algebraic connectivity of $\mathcal{G}$)\\ 
		$\eta$ & Bound on the gradient of $f(\cdot)$\\
		$\kappa$ & Strong convexity parameter\\
		$\mu$ & Lipschitz constant\\
		\bottomrule
	\end{tabular}
\end{table}

\subsection{Proof of Lemma \ref{opt lemma1}}
\label{opt App:lemma1}

Since $A=A^T$ and $g(\cdot)$ is an odd function,
\begin{multline}  
d{\overline{y}_{t}} =\frac{1}{N}\sum_{i=1}^{N}d{y_{i,t}} 
 =\frac{1}{N}\sum_{i=1}^{N}\left[ -\nabla f(y_{i,t})\tilde{\gamma}dt+\sum_{j=1,j\neq i}^N \alpha_{ij}g(y_{i,t}-y_{j,t})\tilde{\gamma}dt+\tau \tilde{\gamma}dB_{i,t}\right]  \\
 =-\frac{1}{N}\sum_{i=1}^{N}\nabla f(y_{i,t})\tilde{\gamma}dt+\frac{\tau }{N}\tilde{\gamma}%
\sum_{i=1}^{N}{dB_{i,t}}. \label{opt eq:macro}
\end{multline}%
%It follows that
By (\ref{opt eq: dy_it}) and (\ref{opt eq:macro}),
\begin{align}
\label{opt eq:micro} 
d{e_{i,t}}=& d{y}_{i,t}-d{\overline{y}_{t}}  \notag  \\
=& \sum_{j=1,j\neq i}^{N}\alpha_{ij}g(y_{i,t}-y_{j,t})\tilde{\gamma}dt-\nabla f(y_{i,t})\tilde{\gamma}dt+\frac{1}{N}%
\sum_{i=1}^N\nabla f({x}_{i,t})\tilde{\gamma}dt  +\tau \tilde{\gamma}{dB_{i,t}}-\frac{\tau }{N}\tilde{\gamma}\sum_{i=1}^N {dB_{i,t}}.
\end{align}%
From equation (\ref{opt eq: g()}) and the assumption that $g_a(\|x\|) = a$,
\begin{align*}
g(y_{i,t}-y_{j,t}) =-(y_{i,t}-y_{j,t})[g_{a}(\Vert y_{i,t}-y_{j,t}\Vert
)-g_{r}(\Vert y_{i,t}-y_{j,t}\Vert )] =-(y_{i,t}-y_{j,t})[a-g_{r}(\Vert
y_{i,t}-y_{j,t}\Vert )].
\end{align*}%
We then have 
\begin{align*}
\sum_{j=1,j\neq i}^{N}\alpha_{ij}g(y_{i,t}-y_{j,t})= -a\sum_{j=1,j\neq
	i}^{N}\alpha_{ij}(y_{i,t}-y_{j,t})+\sum_{j=1,j\neq i}^{N}\alpha_{ij}g_{r}(\Vert
y_{i,t}-y_{j,t}\Vert )(y_{i,t}-y_{j,t}) .
\end{align*}%
Consequently, equation (\ref{opt eq:micro}) becomes 
\begin{multline}
\label{opt eq:micro2}
d{e_{i,t}}= -a\sum_{j=1,j\neq
	i}^{N}\alpha_{ij}(y_{i,t}-y_{j,t})\tilde{\gamma}dt +\sum_{j=1,j\neq i}^{N}\alpha_{ij}g_{r}(\Vert
y_{i,t}-y_{j,t}\Vert )(y_{i,t}-y_{j,t})\tilde{\gamma}dt
-\nabla f(y_{i,t})\tilde{\gamma}dt\\
+\frac{1}{N}\sum_{j=1}^N\nabla f({x}_{j,t}) \tilde{\gamma}dt
+\tau \tilde{\gamma}{dB_{i,t}}-\frac{\tau }{N}\tilde{\gamma}\sum_{j=1}^N{dB_{j,t}}.
\end{multline}%
In light of the facts that $dt\cdot dt=dt\cdot dB_{i,t}=0$, $dB_{i,t}\cdot dB_{j,t}=0$ ($i\neq j$), and $dB_{i,t}\cdot dB_{i,t}=mdt$ (see \cite{oksendal2003stochastic}),
	\begin{align*}
	de_{i,t}\cdot de_{i,t}=\left( 1-\frac{1}{N}\right)m\tau^2 \tilde{\gamma}^2dt.
	\end{align*}
	Then from Ito's lemma,
\begin{align}  \label{opt eq: dV_i}
d{V_{i,t}}=& d{e_{i,t}}\cdot e_{i,t}+\frac{1}{2}de_{i,t}\cdot de_{i,t} 
\notag \\
=& -a\sum_{j=1,j\neq
	i}^{N}\alpha_{ij}(y_{i,t}-y_{j,t})^T e_{i,t}\tilde{\gamma}dt+\sum_{j=1,j\neq i}^{N}\alpha_{ij}g_{r}(\Vert
y_{i,t}-y_{j,t}\Vert )(y_{i,t}-y_{j,t})^{T}e_{i,t}\tilde{\gamma}dt \notag\\
& -\left[ \nabla
f(y_{i,t})-\frac{1}{N}\sum_{j=1}^N\nabla f({x}_{j,t})\right] ^{T}e_{i,t}\tilde{\gamma}dt  +\left[ \tau \tilde{\gamma}{dB_{i,t}}-\frac{\tau }{N}\tilde{\gamma}\sum_{j=1}^N{dB_{j,t}}\right]
^{T}e_{i,t}+\frac{m\tau ^{2} \tilde{\gamma}^2}{2}\left( 1-\frac{1}{N}\right) dt.
\end{align}
In conclusion, the stochastic differential equation is 
\begin{align*}
%\label{opt eq:doV2}
d\overline{V}_t =& \frac{1}{N}\sum_{i=1}^{N} dV_{i,t}  \notag \\
= & -\frac{a}{N}\sum_{i=1}^{N}\sum_{j=1,j\neq
	i}^{N}\alpha_{ij}(y_{i,t}-y_{j,t})^T e_{i,t}\tilde{\gamma}dt+ \frac{1}{N}\sum_{i=1}^{N}\sum_{j=1,j\neq i}^{N}\alpha_{ij}g_{r}(\Vert
y_{i,t}-y_{j,t}\Vert )(y_{i,t}-y_{j,t})^{T}e_{i,t}\tilde{\gamma}dt \notag\\
&-\frac{1}{N}%
\sum_{i=1}^N \nabla^T f(y_{i,t}) e_{i,t} \tilde{\gamma}dt  + \frac{\tau}{N} \tilde{\gamma}\sum_{i=1}^N {dB_{i,t}}^T e_{i,t} +\frac{m\tau^2 \tilde{\gamma}^2(N-1)%
}{2N}dt \notag \\
= & -\frac{a}{N}\sum_{i=1}^{N}\sum_{j=1,j\neq
	i}^{N}\alpha_{ij}(e_{i,t}-e_{j,t})^T e_{i,t}\tilde{\gamma}dt+ \frac{1}{2N}\sum_{i=1}^{N}\sum_{j=1,j\neq i}^{N}\alpha_{ij}g_{r}(\Vert
y_{i,t}-y_{j,t}\Vert )\Vert y_{i,t}-y_{j,t} \Vert ^2 \tilde{\gamma}dt\notag\\
&-\frac{1}{N}%
\sum_{i=1}^N \nabla^T f(y_{i,t}) e_{i,t} \tilde{\gamma}dt  + \frac{\tau}{N} \tilde{\gamma}\sum_{i=1}^N {dB_{i,t}}^T e_{i,t}+\frac{m\tau^2 \tilde{\gamma}^2(N-1)%
}{2N}dt. 
\end{align*}

\subsection{Proof of Theorem \ref{opt thm1}}
\label{opt App:thm1}

\subsubsection{Case 1}
By equation (\ref{opt eq:doV}), 
\begin{align*} 
d\overline{V}_t
= & -\frac{a}{N}\sum_{i=1}^{N}\sum_{j=1,j\neq
	i}^{N}\alpha_{ij}(e_{i,t}-e_{j,t})^T e_{i,t}\tilde{\gamma}dt+ \frac{1}{2N}\sum_{i=1}^{N}\sum_{j=1,j\neq i}^{N}\alpha_{ij}g_{r}(\Vert
y_{i,t}-y_{j,t}\Vert )\Vert y_{i,t}-y_{j,t} \Vert ^2 \tilde{\gamma}dt \notag\\
& -\frac{1}{N}%
\sum_{i=1}^N \nabla^T f(y_{i,t}) e_{i,t} \tilde{\gamma}dt + \frac{\tau}{N} \tilde{\gamma}\sum_{i=1}^N {dB_{i,t}}^T e_{i,t}+\frac{m\tau^2 \tilde{\gamma}^2(N-1)%
}{2N}dt. 
\end{align*}
%\begin{multline*}
%d\overline{V}_t =  -\frac{a}{N}\sum_{i=1}^{N}\sum_{j=1,j\neq
%	i}^{N}\alpha_{ij}(e_{i,t}-e_{j,t})^T e_{i,t}\tilde{\gamma}dt\\
%+ \frac{1}{2N}\sum_{i=1}^{N}\sum_{j=1,j\neq i}^{N}\alpha_{ij}g_{r}(\Vert
%y_{i,t}-y_{j,t}\Vert )\Vert y_{i,t}-y_{j,t} \Vert ^2 \tilde{\gamma}dt \\
%-\frac{1}{N}%
%\sum_{i=1}^N \nabla^T f(y_{i,t}) e_{i,t} \tilde{\gamma}dt + \frac{\tau}{N} \tilde{\gamma}\sum_{i=1}^N {dB_{i,t}}^T e_{i,t}+\frac{m\tau^2 \tilde{\gamma}^2(N-1)%
%}{2N}dt.
%\end{multline*}
Noticing that $g_{r}(\Vert
y_{i,t}-y_{j,t}\Vert )\Vert y_{i,t}-y_{j,t} \Vert ^2\le b$, by (\ref{opt eq: sum1}) and (\ref{opt eq: sum2}),
\begin{align*}
d\overline{V}_t
\le & -2{a}\lambda_2\tilde{\gamma}\overline{V}_t dt+\frac{b | Tr(L) |}{2N}\tilde{\gamma}dt+\frac{m\tau^2\tilde{\gamma}^2(N-1)}{2N}dt +\frac{1}{N}\sum_{i=1}^N \eta \|e_{i,t}\| \tilde{\gamma}dt+\frac{\tau}{N}\tilde{\gamma} \sum_{i=1}^N {%
	dB_{i,t}}^T e_{i,t}  \notag \\
\le & -2{a}\lambda_2\tilde{\gamma}\overline{V}_t dt+ \left[\frac{b | Tr(L) |}{2N}\tilde{\gamma}+\frac{m\tau^2 \tilde{\gamma}^2(N-1)}{2N}%
\right]dt+\sqrt{2}{\eta}\sqrt{\overline{V}_t}\tilde{\gamma}dt+\frac{\tau}{N}\tilde{\gamma} \sum_{i=1}^N {%
	dB_{i,t}}^T e_{i,t}  \notag \\
\le & -c_1\tilde{\gamma}\overline{V}_t dt-(2a\lambda_2-c_1)\tilde{\gamma}\overline{V}_t dt+\sqrt{2}{\eta}\sqrt{%
	\overline{V}_t}\tilde{\gamma}dt  +\left[\frac{b | Tr(L) |}{2N}\tilde{\gamma}+\frac{m\tau^2\tilde{\gamma}^2(N-1)}{2N}\right]dt 
\notag \\
& +\frac{\tau}{N}\tilde{\gamma} \sum_{i=1}^N {dB_{i,t}}^T e_{i,t}  \notag \\
= & -c_1\tilde{\gamma}\overline{V}_t dt-(2a\lambda_2-c_1)\left[\sqrt{\overline{V}}-\frac{\sqrt{2}{%
		\eta}}{2(2a\lambda_2-c_1)}\right]^2\tilde{\gamma}dt   \\
&+\left[\frac{\eta^2}{2(2a\lambda_2-c_1)}\tilde{\gamma}+\frac{b | Tr(L) |}{2N}\tilde{\gamma}+\frac{m\tau^2\tilde{\gamma}^2(N-1)}{2N}%
\right]dt +\frac{\tau}{N}\tilde{\gamma} \sum_{i=1}^N {dB_{i,t}}^T e_{i,t}  \notag \\
\le & -c_1\tilde{\gamma}\overline{V}_t dt+c_2\tilde{\gamma}dt+\frac{\tau}{N}\tilde{\gamma} \sum_{i=1}^N {dB_{i,t}}^T
e_{i,t},  \notag \\
\end{align*}
where $c_1\in \left(0,2a\lambda_2\right)$ is arbitrary, and
\begin{align*}
c_2=\frac{\eta^2}{2(2a\lambda_2-c_1)}+\frac{b | Tr(L) |}{2N}+\frac{m\tau^2\tilde{\gamma}(N-1)}{2N}.
\end{align*}
Applying Ito's lemma to $e^{c_1\tilde{\gamma} t}\overline{V}_t$, 
\begin{align*}
d(e^{c_1\tilde{\gamma} t}\overline{V}_t)=  e^{c_1\tilde{\gamma} t}d \overline{V}_t+c_1\tilde{\gamma}e^{c_1\tilde{\gamma} t}%
\overline{V}_t dt 
\le c_2\tilde{\gamma}e^{c_1\tilde{\gamma} t}dt+ \frac{\tau}{N} \tilde{\gamma}e^{c_1\tilde{\gamma} t} \sum_{i=1}^N {dB_{i,t}}^T
e_{i,t}.
\end{align*}
Integrating the stochastic differential inequality, 
\begin{align*}
\overline{V}_t \le  e^{-c_1\tilde{\gamma} t}\overline{V}_0+\frac{c_2}{c_1}(1-e^{-c_1\tilde{\gamma}
	t})
+e^{-c_1\tilde{\gamma} t}\int_{0}^{t}\frac{\tau}{N} \tilde{\gamma}e^{c_1\tilde{\gamma} s} \sum_{i=1}^N {dB_{i,s}}^T
e_{i,s}.
\end{align*}
Taking an ensemble average on both sides yields 
\begin{align}
\mathbb{E}[\overline{V}_t ] \le e^{-c_1\tilde{\gamma} t}\overline{V}_0+\frac{c_2}{c_1}%
(1-e^{-c_1\tilde{\gamma} t}).
\end{align}
It follows that 
\begin{align*}
\mathbb{E}[\overline{V}_t ] \le (\overline{V}_0-\frac{c_2}{c_1})e^{-c_1\tilde{\gamma}
	t}+\frac{c_2}{c_1} 
\le \left\{ 
\begin{array}{lcc}
\overline{V}_0 \ \  & \text{if} \ \ \overline{V}_0\ge \cfrac{c_2}{c_1}, & 
\\ 
\cfrac{c_2}{c_1} \ \  & \text{if} \ \ \overline{V}_0< \cfrac{c_2}{c_1}. & 
\end{array}
\right.
\end{align*}
Therefore 
\begin{align}
\mathbb{E}[\overline{V}_t ] \le \max \{\overline{V}_0,\frac{c_2}{c_1}%
\},\forall t,
\end{align}
In the long run, 
\begin{align}  \label{opt neq: oV c1c2}
\mathbb{E}[\overline{V}_t ] \le \frac{c_2}{c_1}=\frac{\eta^2}{%
	2c_1(2a\lambda_2-c_1)}+\frac{b | Tr(L) |}{2c_1 N}+\frac{m\tau^2\tilde{\gamma}(N-1)}{2c_1 N}.
\end{align}
Notice that the above inequality is valid for all $c_1\in (0,2a\lambda_2)$, of which
we look for the minimum over all possible $c_1$'s. Define 
\begin{multline*}
\psi_1(c_1)= \frac{\eta^2}{%
	2c_1(2a\lambda_2-c_1)}+\frac{b | Tr(L) |}{2c_1 N}+\frac{m\tau^2\tilde{\gamma}(N-1)}{2c_1 N}  \\
= \left[\frac{\eta^2}{4a\lambda_2}+\frac{b | Tr(L) |}{2N}+\frac{m\tau^2\tilde{\gamma}(N-1)}{2N} \right]%
\frac{1}{c_1}+\frac{\eta^2}{4a\lambda_2}\frac{1}{(2a\lambda_2-c_1)}.
\end{multline*}
By Cauchy-Schwartz inequality, when 
\begin{align*}
\frac{1}{c_1}=\frac{1}{2a\lambda_2}\left(1+\sqrt{\frac{c_4}{c_3}}\right),
\end{align*}
$\psi_1(c_1)$ attains its minimum 
\begin{align*}
\psi_1^*=\frac{1}{2a\lambda_2}{\left(\sqrt{c_3}+\sqrt{c_4}\right)^2}.
\end{align*}
Here 
\begin{align*}
c_3=\frac{\eta^2}{4a\lambda_2}+\frac{b | Tr(L) |}{2N}+\frac{m\tau^2\tilde{\gamma}(N-1)}{2N}, c_4=\frac{%
	\eta^2}{4a\lambda_2}.
\end{align*}
Since (\ref{opt neq: oV c1c2}) is valid for all $c_1\in (0,2a\lambda_2)$, it holds true
that 
$
\mathbb{E}[\overline{V}_t ] \le \psi_1^*
$
in the long run.

\subsubsection{Case 2}
%
%	\begin{align*}
%	d\overline{V}_t & \nonumber\\
%	= & -\frac{a}{N}\sum_{i=1}^{N}\sum_{j=1,j\neq
%		i}^{N}\alpha_{ij}(e_{i,t}-e_{j,t})^T e_{i,t}\tilde{\gamma}dt+ \frac{1}{2N}\sum_{i=1}^{N}\sum_{j=1,j\neq i}^{N}\alpha_{ij}g_{r}(\Vert
%	y_{i,t}-x_{j,t}\Vert )\Vert y_{i,t}-y_{j,t} \Vert ^2 \tilde{\gamma}dt \notag\\
%	& -\frac{1}{N}%
%	\sum_{i=1}^N \nabla^T f(y_{i,t}) e_{i,t} \tilde{\gamma}dt + \frac{\tau}{N} \tilde{\gamma}\sum_{i=1}^N {dB_{i,t}}^T e_{i,t}+\frac{m\tau^2 \tilde{\gamma}^2(N-1)%
%	}{2N}dt. 
%	\end{align*}
%\begin{multline} 
%d\overline{V}_t 
%=  -\frac{a}{N}\sum_{i=1}^{N}\sum_{j=1,j\neq
%	i}^{N}\alpha_{ij}(e_{i,t}-e_{j,t})^T e_{i,t}\tilde{\gamma}dt-\frac{1}{N}%
%\sum_{i=1}^N \nabla^T f(y_{i,t}) e_{i,t} \tilde{\gamma}dt + \frac{\tau}{N} \tilde{\gamma}\sum_{i=1}^N {dB_{i,t}}^T e_{i,t}\\
%+\frac{m\tau^2\tilde{\gamma}^2(N-1)%
%}{2N}dt.
%\end{multline}
Notice that 
\begin{align*}
\sum_{i=1}^N \nabla^T f(y_{i,t}) e_{i,t} = \sum_{i=1}^N (\nabla^T
f(y_{i,t})-\nabla^T f(\overline{x}_t)) (y_{i,t}-\overline{x}_t)
\geq \sum_{i=1}^N \kappa \|y_{i,t}-\overline{x}_t\|^2 = 2\kappa N\overline{%
	V}_t,
\end{align*}
where the inequality follows from Assumption \ref{opt asp:gradient_strconvexity}, and that $g_{r}(\Vert
y_{i,t}-y_{j,t}\Vert )\Vert y_{i,t}-y_{j,t} \Vert ^2\le b$. In light of (\ref{opt eq: sum1}) and (\ref{opt eq: sum2}), equation (\ref{opt eq:doV}) gives
\begin{align*} 
d\overline{V}_t \le -2{a}\lambda_2\tilde{\gamma}\overline{V}_tdt-2\kappa \tilde{\gamma}\overline{V}_t dt+\frac{b | Tr(L) |}{2N}\tilde{\gamma}dt+\frac{m\tau^2\tilde{\gamma}^2 (N-1)}{2N}dt+\frac{\tau}{N}\tilde{\gamma} \sum_{i=1}^N {%
	dB_{i,t}}^T e_{i,t}.
\end{align*}
By Ito's lemma, 
\begin{align*}
d\left[ e^{2(\kappa+a\lambda_2)\tilde{\gamma}t}\oV_t \right] & =e^{2(\kappa+a\lambda_2)\tilde{\gamma}t}d \oV_t+2(\kappa+a\lambda_2)\tilde{\gamma}e^{2(\kappa+a\lambda_2)\tilde{\gamma}t}\oV_t dt\\
& \le \left[\frac{b | Tr(L) |}{2N}\tilde{\gamma}dt+\frac{m\tau^2\tilde{\gamma}^2 (N-1)}{2N}\right]e^{2(\kappa+a\lambda_2)\tilde{\gamma}t}dt+\frac{\tau}{N}\tilde{\gamma}e^{2(\kappa+a\lambda_2)\tilde{\gamma}t} \sum_{i=1}^N {dB_{i,t}}^T e_{i,t} .
\end{align*}
Integrating both sides,
\begin{multline*}
\oV_t \le  e^{-2(\kappa+a\lambda_2)\tilde{\gamma}t}\oV_0+\left[\frac{b | Tr(L) |}{4N(\kappa+a\lambda_2)}+\frac{m\tau^2\tilde{\gamma} (N-1)}{4N(\kappa+a\lambda_2)}\right]\left[1-e^{-2(\kappa+a\lambda_2)\tilde{\gamma}t}\right]\\
+e^{-2(\kappa+a\lambda_2)\tilde{\gamma}t}\int\limits_{0}^{t}\frac{\tau}{N}\tilde{\gamma}e^{2(\kappa+a\lambda_2)\tilde{\gamma}s} \sum_{i=1}^N {dB_{i,s}^T} e_{i,s}.
\end{multline*}
Taking ensemble average yields
\begin{align*}
%											\label{social neq:doV temp asp_con}
\mathbb{E}[\oV_t] \le e^{-2(\kappa+a\lambda_2)\tilde{\gamma}t}\oV_0+\left[\frac{b | Tr(L) |}{4N(\kappa+a\lambda_2)}+\frac{m\tau^2\tilde{\gamma} (N-1)}{4N(\kappa+a\lambda_2)}\right]\left[1-e^{-2(\kappa+a\lambda_2)\tilde{\gamma}t}\right].
\end{align*}
In the long run,
\begin{align*}
%											\label{social neq:doV long asp_con}
\mathbb{E}[\overline{V}_t] \le \psi_2^* = \frac{b | Tr(L) |}{4N(\kappa+a\lambda_2)}+\frac{m\tau^2\tilde{\gamma} (N-1)}{4N(\kappa+a\lambda_2)}.
\end{align*}

\subsection{Proof of Lemma \ref{opt thm2}}
\label{opt App:thm2}
\subsubsection{Preliminaries}

According to equation (\ref{opt eq:macro}) and Ito's lemma, 
\begin{align}  \label{opt eq:dU}
d{U}_t = & d({\overline{y}_t-x^{\ast}})\cdot (\overline{y}_t-x^{\ast})+\frac{%
	1}{2}d(\overline{y}_t-x^{\ast})\cdot d(\overline{y}_t-x^{\ast})  \notag \\
= & -\frac{1}{N}\sum_{i=1}^{N}\nabla^T f(y_{i,t})(\overline{y}_t-x^{\ast})\tilde{\gamma}
dt +\frac{\tau}{N} \tilde{\gamma}\sum_{i=1}^N {dB_{i,t}}^T (\overline{y}_t-x^{\ast})+\frac{m\tau^2\tilde{\gamma}^2%
}{2N}dt  \notag \\
= & -\frac{1}{N}\sum_{i=1}^N \nabla^T f(y_{i,t}) (y_{i,t}-x^{\ast}) \tilde{\gamma}dt  +\frac{1}{N}\sum_{i=1}^N \nabla^T f(y_{i,t}) e_{i,t} \tilde{\gamma}dt+\frac{\tau}{N} \tilde{\gamma}\sum_{i=1}^N 
{dB_{i,t}}^T (\overline{y}_t-x^{\ast})+\frac{m\tau^2\tilde{\gamma}^2}{2N}dt.
\end{align}
Since $f(\cdot)$ attains its minimum at $x^{\ast}$, $\nabla f(x^{\ast})=0$.
Then by Assumption \ref{opt asp:gradient_strconvexity} and equation (\ref{opt eq: FUV}), 
\begin{align*}
-\frac{1}{N}\sum_{i=1}^N \nabla^T f(y_{i,t}) (y_{i,t}-x^{\ast})=&-\frac{1}{N}%
\sum_{i=1}^N (\nabla^T f(y_{i,t})-\nabla^T f(x^{\ast})) (y_{i,t}-x^{\ast})\\
\leq &-2\kappa \overline{F}_t
= -2\kappa(U_t+\overline{V}_t).
\end{align*}
%By Assumption \ref{opt asp:gradient_strconvexity} and equation (\ref{opt eq: FUV}),
%\begin{align*}
%-\frac{1}{N}\sum_{i=1}^N (\nabla^T f(y_{i,t})-\nabla^T f(x^{\ast})) (y_{i,t}-x^{\ast})\leq -2\kappa \oE_t = -2\kappa(U_t+\oV_t).
%\end{align*}
By Assumption \ref{opt asp:Lipschitz}, 
\begin{align*}
\frac{1}{N}\sum_{i=1}^N \nabla^T f(y_{i,t}) e_{i,t}=\frac{1}{N}\sum_{i=1}^N
\left(\nabla^T f(y_{i,t})-\nabla^T f(\overline{x}_t)\right) (y_{i,t}-%
\overline{x}_t)\leq \frac{1}{N}\sum_{i=1}^N \mu \|e_{i,t}\|^2= 2\mu\overline{V%
}_t.
\end{align*}
Therefore, 
\begin{align}  \label{opt neq: dU}
dU_t & \leq -2\kappa\tilde{\gamma}(U_t+\overline{V}_t)dt+2\mu\tilde{\gamma}\overline{V}_tdt+\frac{%
	m\tau^2\tilde{\gamma}^2}{2N}dt+\frac{\tau}{N} \tilde{\gamma}\sum_{i=1}^N {dB_{i,t}}^T \overline{x}_t 
\notag \\
& = -2\kappa\tilde{\gamma} U_t dt+2(\mu-\kappa)\tilde{\gamma}\overline{V}_t dt+\frac{m\tau^2\tilde{\gamma}^2}{2N}dt+%
\frac{\tau}{N}\tilde{\gamma} \sum_{i=1}^N {dB_{i,t}}^T \overline{x}_t.
\end{align}

%\begin{lemma}
%	\label{opt lem}
%	{\color{blue}Suppose relation (\ref{opt eq: dy_it}) and Assumption \ref{opt asp:gradient_strconvexity} hold. Then $\overline{V}_t$ satisfies}
%	\begin{align}  \label{opt neq:doV}
%	d\overline{V}_t \le & -2{a}\lambda_2\tilde{\gamma}\overline{V}_t dt-2\kappa \tilde{\gamma}\overline{V}_t dt+\frac{m\tau^2\tilde{\gamma}^2 (N-1)}{2N}dt+\frac{\tau}{N} \tilde{\gamma}\sum_{i=1}^N {%
%		dB_{i,t}}^T e_{i,t}.
%	\end{align}
%\end{lemma}
%
%\proof{Proof of Lemma \ref{opt lem}.}
By equation (\ref{opt eq:doV_simp}), 
\begin{multline} 
d\overline{V}_t 
=  -\frac{a}{N}\sum_{i=1}^{N}\sum_{j=1,j\neq
	i}^{N}\alpha_{ij}(e_{i,t}-e_{j,t})^T e_{i,t}\tilde{\gamma}dt-\frac{1}{N}%
\sum_{i=1}^N \nabla^T f(y_{i,t}) e_{i,t} \tilde{\gamma}dt + \frac{\tau}{N} \tilde{\gamma}\sum_{i=1}^N {dB_{i,t}}^T e_{i,t}\\
+\frac{m\tau^2\tilde{\gamma}^2(N-1)%
}{2N}dt.
\end{multline}
Notice that 
\begin{align*}
\sum_{i=1}^N \nabla^T f(y_{i,t}) e_{i,t} = \sum_{i=1}^N (\nabla^T
f(y_{i,t})-\nabla^T f(\overline{x}_t)) (y_{i,t}-\overline{x}_t)
\geq \sum_{i=1}^N \kappa \|y_{i,t}-\overline{x}_t\|^2 = 2\kappa N\overline{%
	V}_t,
\end{align*}
where the inequality follows from Assumption \ref{opt asp:gradient_strconvexity}%
. In light of (\ref{opt eq: sum1}) and (\ref{opt eq: sum2}), 
\begin{align} 
\label{opt neq:doV}
d\overline{V}_t \le -2{a}\lambda_2\tilde{\gamma}\overline{V}_tdt-2\kappa \tilde{\gamma}\overline{V}_t dt+\frac{m\tau^2\tilde{\gamma}^2 (N-1)}{2N}dt+\frac{\tau}{N}\tilde{\gamma} \sum_{i=1}^N {%
	dB_{i,t}}^T e_{i,t}.
\end{align}
%\Halmos
%\endproof

\subsubsection{Proof of Lemma  \ref{opt thm2}}

Define 
\begin{align}  \label{opt def:W}
W_t=U_t+\frac{(\mu-\kappa)}{a\lambda_2}\overline{V}_t.
\end{align}
By (\ref{opt neq: dU}) and (\ref{opt neq:doV}), 
\begin{align*}
dW_t \leq & -2\kappa\tilde{\gamma} U_t dt+\left[2(\mu-\kappa)-\frac{(\mu-\kappa)}{a\lambda_2}%
(2\kappa+2a\lambda_2)\right]\tilde{\gamma}\overline{V}_t dt+\frac{m\tau^2\tilde{\gamma}^2}{2N}dt+\frac{m\tau^2\tilde{\gamma}^2(\mu-\kappa)(N-1)}{2aN\lambda_2}dt\\ 
&
+\frac{\tau}{N}\tilde{\gamma} \sum_{i=1}^N {dB_{i,t}}^T \overline{x}_t+\frac{%
	\tau(\mu-\kappa)}{aN\lambda_2}\tilde{\gamma}\sum_{i=1}^N {dB_{i,t}}^T e_{i,t} \\
= & -2\kappa\tilde{\gamma} W_t dt +\frac{m\tau^2\tilde{\gamma}^2}{2N}\left[1+\frac{(\mu-\kappa)(N-1)}{a\lambda_2}%
\right]dt+\frac{\tau}{N}\tilde{\gamma}
\sum_{i=1}^N {dB_{i,t}}^T \overline{x}_t +\frac{%
	\tau(\mu-\kappa)}{aN\lambda_2}\tilde{\gamma} \sum_{i=1}^N {dB_{i,t}}^T e_{i,t}.
\end{align*}
Then, 
\begin{align*}
d(e^{2\kappa\tilde{\gamma} t}W_t)&=  e^{2\kappa\tilde{\gamma} t}dW_t+2\kappa\tilde{\gamma} e^{2\kappa\tilde{\gamma} t}W_t dt \\
&\leq  \frac{m\tau^2\tilde{\gamma}^2}{2N}\left[1+\frac{(\mu-\kappa)(N-1)}{a\lambda_2}%
\right]e^{2\kappa\tilde{\gamma} t}dt +\frac{\tau}{N}\tilde{\gamma} e^{2\kappa\tilde{\gamma} t}\left[\sum_{i=1}^N {dB_{i,t}}^T \overline{x}%
_t+\frac{(\mu-\kappa)}{a\lambda_2} \sum_{i=1}^N {dB_{i,t}}^T e_{i,t}\right].
\end{align*}
Integrating both sides yields 
\begin{multline*}
W_t \leq  e^{-2\kappa\tilde{\gamma} t}W_0+\frac{m\tau^2\tilde{\gamma}}{4\kappa N}\left[1+\frac{(\mu-\kappa)(N-1)}{a\lambda_2}%
\right]
(1-e^{-2\kappa\tilde{\gamma} t}) \\
+e^{-2\kappa\tilde{\gamma} t}\int_{0}^{t}\frac{\tau}{N}\tilde{\gamma} e^{2\kappa\tilde{\gamma} s}\left[%
\sum_{i=1}^N {dB_{i,s}}^T \overline{x}_s+\frac{(\mu-\kappa)}{a\lambda_2}
\sum_{i=1}^N {dB_{i,s}}^T e_{i,s}\right].
\end{multline*}
Taking ensemble average, we get 
\begin{align*}
\mathbb{E}[W_t] \leq e^{-2\kappa\tilde{\gamma} t}W_0+\frac{m\tau^2\tilde{\gamma}}{4\kappa N}\left[1+\frac{(\mu-\kappa)(N-1)}{a\lambda_2}%
\right](1-e^{-2\kappa\tilde{\gamma} t}).
\end{align*}
By (\ref{opt def:W}),
\begin{align*}
\mathbb{E}[U_t] \leq \mathbb{E}[ W_t]
\leq e^{-2\kappa\tilde{\gamma} t}\left[U_0+\frac{(\mu-\kappa)}{a\lambda_2}\overline{V}_0\right] +\frac{m\tau^2\tilde{\gamma}}{4\kappa N}\left[1+\frac{(\mu-\kappa)(N-1)}{a\lambda_2}%
\right](1-e^{-2\kappa\tilde{\gamma} t}).
\end{align*}
In the long run, 
\begin{align}
\mathbb{E}[U_t]  \leq \frac{m\tau^2\tilde{\gamma}}{4\kappa N}\left[1+\frac{(\mu-\kappa)(N-1)}{a\lambda_2}%
\right].
\end{align}

\subsection{Proof of Proposition \ref{opt prop 1}}
\label{opt App:prop1}
By (\ref{opt neq: EUx}) and (\ref{opt neq: EG lowerx}),
\begin{equation*}
\frac{1}{2}\mathbb{E}[\Vert x_{t}-x^{\ast }\Vert ^{2}]  \geq e^{-2\mu\gamma\Gamma t}G_{0}+\frac{m\tau _{N}^{2}\gamma}{4\mu}%
(1-e^{-2\mu\gamma\Gamma t}), 
\end{equation*}%
and
\begin{align*}
\frac{1}{2}\mathbb{E}[\Vert \bar{x}_{t}-x^{\ast }\Vert ^{2}]   \leq e^{-2\kappa\tilde{\gamma}\tilde{\Gamma} t}\left[U_0+\frac{(\mu-\kappa)}{a\lambda_2}\overline{V}_0\right] +\frac{m\tau^2\tilde{\gamma}}{4\kappa N}\left[1+\frac{(\mu-\kappa)(N-1)}{a\lambda_2}%
\right](1-e^{-2\kappa\tilde{\gamma}\tilde{\Gamma} t}).
\end{align*}
Let \[d_1=\frac{m\tau _{N}^{2}\gamma}{4\mu},\ \ d_2=U_0+\frac{(\mu-\kappa)}{a\lambda_2}\overline{V}_0,\]
and
\[d_3=\frac{m\tau^2\tilde{\gamma}}{4\kappa N}\left[1+\frac{(\mu-\kappa)(N-1)}{a\lambda_2}%
\right].\]
It follows that
\begin{equation*}
\mathbb{E}[\Vert x_{t}-x^{\ast }\Vert ^{2}]\geq 2(G_{0}-d_1)e^{-2\mu\gamma \Gamma t}+2d_1,
\end{equation*}%
and
\begin{align*}
\mathbb{E}[\Vert \bar{x}_{t}-x^{\ast }\Vert ^{2}]  \leq 2(d_2-d_3)e^{-2\kappa\tilde{\gamma}\tilde{\Gamma} t} +2d_3.
\end{align*}
From the discussion in Section \ref{opt subsec: real}, $d_1\sim 1/N^{(1/\beta)-1}$ and $d_3\sim 1/N$. Therefore,
there exists $N^*$ such that when $N>N^*$, $d_3<d_1$. In this case, let $t^*$ be such that
\[d_1-d_3=|G_0-d_1|e^{-2\mu\gamma\Gamma t^*}+|d_2-d_3|e^{-2\kappa \tilde{\gamma}\tilde{\Gamma} t^*}.\]
Then for all $t>t^*$,
\begin{align*}
\mathbb{E}[\Vert \bar{x}_{t}-x^{\ast }\Vert ^{2}] < \mathbb{E}[\Vert x_{t}-x^{\ast }\Vert ^{2}].
\end{align*}

\subsection{Proof of Theorem \ref{opt Thm: nonconvex}}
\label{opt App:thm3}

We start by looking for the joint density of $(\oy^T,e_1^T,\dots,e_{N-1}^T)$. Notice that $(\oy^T,e_1^T,\dots,e_{N-1}^T)=(D\otimes I_m)\my_t$, where $D$ is a $N\times N$ matrix. It is easy to verify that $D$ has full rank, so that $(D\otimes I_m)^{-1}$ exists.  It follows that (see \cite{jacod2003probability})
\begin{align*}
\pi(\oy^T,e_1^T,\dots,e_{N-1}^T) & =\frac{1}{\det(D\otimes I_m)}\pi((D\otimes I_m)^{-1}(\oy^T,e_1^T,\dots,e_{N-1}^T))\\
& = \frac{1}{\tilde{K}}\exp\left\{-2\left[\sum_{i=1}^{N}f(\oy+e_i)+\frac{1}{2}\sum_{i=1}^{N}\sum_{j=1,j\neq i}^N \alpha_{ij} J(\|e_{i,t}-e_{j,t}\|)\right]/(\tau^2\tilde{\gamma})\right\}.
\end{align*}
%\begin{align*}
%\pi_{\oy,e_1,\dots,e_{N-1}}(\oy,e_1,\dots,e_{N-1})=\frac{1}{\tilde{K}}\exp\left\{-2\left[\sum_{i=1}^{N}f(\oy+e_i)+\frac{1}{2}\sum_{i=1}^{N}\sum_{j=1,j\neq i}^N \alpha_{ij} J(\|e_{i,t}-e_{j,t}\|)\right]/(\tau^2\tilde{\gamma})\right\},
%\end{align*}
where $\tilde{K}$ is a normalizing factor.
The density of $\oy$ is calculated as 
\begin{align*}
\pi(\oy)=\frac{1}{\tilde{K}}\int \exp\left\{-2\left[\sum_{i=1}^{N}f(\oy+e_i)+\frac{1}{2}\sum_{i=1}^{N}\sum_{j=1,j\neq i}^N \alpha_{ij} J(\|e_{i,t}-e_{j,t}\|)\right]/(\tau^2\tilde{\gamma})\right\}de_1\cdots de_{N-1}.
\end{align*}
Notice that $e_N=-\sum_{i=1}^{N-1}e_i$ in the equation above.
We rewrite the potential function $J$ as a sum of two parts: $J=J_a+J_r$, where 
$\nabla_x J_a(\|x\|)=x g_a(\|x\|)=ax$, and $\nabla_x J_r(\|x\|)=xg_r(\|x\|)$.
Without loss of generality, we assume that
$ J_a(\|x\|)=({1}/{2})a\|x\|^2$.
Then (refer to Godsil and Royle \cite{godsil2013algebraic})
\begin{align*}
\pi(\oy)
& =\frac{1}{\tilde{K}}\int \exp\left\{-2\left[\sum_{i=1}^{N}f(\oy+e_i)+\frac{1}{2}\sum_{i=1}^{N}\sum_{j=1,j\neq i}^N \alpha_{ij}\frac{a}{2}\|e_i-e_j\|^2+R\right]/(\tau^2\tilde{\gamma})\right\}de_1\cdots de_{N-1}\\
&  =\frac{1}{\tilde{K}}\int \exp\left\{-2\left[\sum_{i=1}^{N}f(\oy+e_i)+\frac{a}{2}\me^T(L\otimes I_m)\me+R\right]/(\tau^2\tilde{\gamma})\right\}de_1\cdots de_{N-1},
\end{align*}
where
\[R=\frac{1}{2}\sum_{i=1}^{N}\sum_{j=1,j\neq i}^N \alpha_{ij} J_r(\|e_{i,t}-e_{j,t}\|).\]
Let $z_i=a^{{1}/{2}}e_i, \forall i$.
\begin{multline*}
 \int \exp\left\{-2\left[\sum_{i=1}^{N}f(\oy+e_i)+\frac{a}{2}\me^T(L\otimes I_m)\me+R\right]/(\tau^2\tilde{\gamma})\right\}de_1\cdots de_{N-1}\\
=  a^{-(N-1)/2}\int \exp\left\{-2\left[\sum_{i=1}^{N}f(\oy+a^{-{1}/{2}}z_i)+\frac{1}{2}\mz^T(L\otimes I_m)\mz+\tilde{R}\right]/(\tau^2\tilde{\gamma})\right\}dz_1\cdots dz_{N-1}.
\end{multline*}
Here 
\[\tilde{R}=\frac{1}{2}\sum_{i=1}^{N}\sum_{j=1,j\neq i}^N \alpha_{ij} J_r(a^{-1/2}\|z_{i,t}-z_{j,t}\|).\]
Given that graph $\mathcal{G}$ is connected, 
$
\mz^T(L\otimes I_m)\mz\ge \lambda_2\me^T \me
$
with $\lambda_2>0$. Then since $f$ and $J_r$ are continuous, we have
\begin{multline*}
 \lim\limits_{a\rightarrow \infty}\int \exp\left\{-2\left[\sum_{i=1}^{N}f(\oy+a^{-{1}/{2}}z_i)+\frac{1}{2}\mz^T(L\otimes I_m)\mz+\tilde{R}\right]/(\tau^2\tilde{\gamma})\right\}dz_1\cdots dz_{N-1}\\
= \int \exp\left\{-2\left[Nf(\oy)+\frac{1}{2}\mz^T(L\otimes I_m)\mz+\tilde{R}_0 \right]/(\tau^2\tilde{\gamma})\right\}dz_1\cdots dz_{N-1}\\
= \exp\{-2Nf(\oy)/(\tau^2\tilde{\gamma})\}\exp\{-2\tilde{R}_0/(\tau^2\tilde{\gamma})\}\int \exp\left\{\mz^T(L\otimes I_m)\mz/(\tau^2\tilde{\gamma})\right\}dz_1\cdots dz_{N-1},
\end{multline*}
where 
\[\tilde{R}_0=\frac{1}{2}\sum_{i=1}^{N}\sum_{j=1,j\neq i}^N \alpha_{ij} J_r(0)=\frac{1}{2}|Tr(L)|J_r(0).\]
Therefore,
\begin{align*}
&\lim\limits_{a\rightarrow \infty}\pi(\oy) \\
&=  \lim\limits_{a\rightarrow \infty}\frac{1}{\tilde{K}}\int\exp\left\{-2\left[\sum_{i=1}^{N}f(\oy+e_i)+\frac{1}{2}\sum_{i=1}^{N}\sum_{j=1,j\neq i}^N \alpha_{ij} J(\|e_{i,t}-e_{j,t}\|)\right]/(\tau^2\tilde{\gamma})\right\}de_1\cdots de_{N-1}\\
&= \frac{\exp\{-2Nf(\oy)/(\tau^2\tilde{\gamma})\}\exp\{-2\tilde{R}_0/\tau^2\}\int \exp\left\{\mz^T(L\otimes I_m)\mz/(\tau^2\tilde{\gamma})\right\}dz_1\cdots dz_{N-1}}{\int \exp\{-2Nf(\oy)/(\tau^2\tilde{\gamma})\}\exp\{-2\tilde{R}_0/(\tau^2\tilde{\gamma})\}\int \exp\left\{\mz^T(L\otimes I_m)\mz/(\tau^2\tilde{\gamma})\right\}dz_1\cdots dz_{N-1}d\oy}\\
&= \frac{\exp\{-2Nf(\oy)/(\tau^2\tilde{\gamma})\}}{\int \exp\{-2Nf(\oy)/(\tau^2\tilde{\gamma})\}d\oy}\\
& =\frac{\exp\left\{-2Nf(\oy)/(\sigma^2\tilde{\Gamma})\right\}}{\int \exp\left\{-2Nf(\oy)/(\sigma^2\tilde{\Gamma})\right\}d\oy}.
\end{align*}
This completes the proof.

\bibliographystyle{plain}
\bibliography{mybib} % if more than one, comma separated
\end{document}